\documentclass[11pt]{article}

\usepackage[utf8]{inputenc}
\usepackage[T1]{fontenc}
\usepackage{iftex}

\usepackage[margin=1in]{geometry}
\usepackage{charter}
\usepackage[square,comma,sort&compress]{natbib}
\usepackage{tabularx}

\usepackage{url}
\usepackage{booktabs}
\usepackage{amsfonts}
\ifPDFTeX
  \ifnum\pdfoutput>0
    \usepackage[protrusion=true,expansion=true,tracking=true]{microtype}
  \else
    \usepackage[protrusion=true,expansion=false,tracking=true]{microtype}
  \fi
\else
  \usepackage[protrusion=true,expansion=true,tracking=true]{microtype}
\fi
\usepackage{nicefrac}
\usepackage{xspace}
\usepackage{bm}
\usepackage{amsmath}
\usepackage{amsthm}
\usepackage{amssymb}
\usepackage{mathtools}
\usepackage[most]{tcolorbox}
\usepackage{xparse}
\usepackage{enumitem}
\usepackage{xcolor}
\usepackage{aliascnt}
\usepackage{hyperref}
\usepackage[capitalize,nameinlink,noabbrev]{cleveref}

\definecolor{darkblue}{rgb}{0.0,0.0,0.45}
\hypersetup{
  colorlinks=true,
  linkcolor=darkblue,
  citecolor=darkblue,
  urlcolor=darkblue
}

\theoremstyle{plain}
\numberwithin{equation}{section}
\newtheorem{theorem}{Theorem}[section]
\numberwithin{theorem}{section}
\newaliascnt{lemma}{theorem}
\newtheorem{lemma}[lemma]{Lemma}
\aliascntresetthe{lemma}
\newaliascnt{proposition}{theorem}
\newtheorem{proposition}[proposition]{Proposition}
\aliascntresetthe{proposition}
\newaliascnt{corollary}{theorem}
\newtheorem{corollary}[corollary]{Corollary}
\aliascntresetthe{corollary}
\newaliascnt{assumption}{theorem}

\aliascntresetthe{assumption}
\theoremstyle{definition}
\newaliascnt{definition}{theorem}
\newtheorem{definition}[definition]{Definition}
\aliascntresetthe{definition}
\newaliascnt{remark}{theorem}
\newtheorem{remark}[remark]{Remark}
\aliascntresetthe{remark}
\newaliascnt{example}{theorem}

\aliascntresetthe{example}

\newif\ifshowresearchboxes
\newif\ifshowmotivationboxes
\newif\ifshowinterpretationboxes
\newif\ifshowverificationboxes
\showresearchboxestrue
\showmotivationboxestrue
\showinterpretationboxestrue
\showverificationboxestrue

\NewDocumentEnvironment{motivation}{+b}{\ifshowresearchboxes
    \ifshowmotivationboxes
      \begin{tcolorbox}[
        breakable,
        colback=blue!4,
        colframe=blue!45!black,
        title=Motivation,
        fonttitle=\bfseries]
      #1
      \end{tcolorbox}
    \fi
  \fi
}{}
\NewDocumentEnvironment{interpretation}{+b}{\ifshowresearchboxes
    \ifshowinterpretationboxes
      \begin{tcolorbox}[
        breakable,
        colback=green!4,
        colframe=green!35!black,
        title=Interpretation,
        fonttitle=\bfseries]
      #1
      \end{tcolorbox}
    \fi
  \fi
}{}
\NewDocumentEnvironment{verification}{+b}{\ifshowresearchboxes
    \ifshowverificationboxes
      \begin{tcolorbox}[
        breakable,
        colback=gray!5,
        colframe=gray!50,
        title=Verification,
        fonttitle=\bfseries]
      #1
      \end{tcolorbox}
    \fi
  \fi
}{}

\DeclarePairedDelimiter{\norm}{\lVert}{\rVert}

\DeclareMathOperator{\conv}{conv}

\newcommand{\R}{\mathbb R}

\newcommand{\ST}{\operatorname{ST}}
\newcommand{\xc}{\operatorname{xc}}
\newcommand{\xcs}{\operatorname{xcs}}

\newcommand{\Sym}{\mathfrak S}

\newcommand{\card}[1]{\left|#1\right|}

\newcommand{\leanverified}{\xspace\textup{[\textsc{Lean}]}\xspace}
\newcommand{\leanassuming}[1][]{\xspace\textup{[\textsc{Lean}, assuming \Cref{thm:small-index-k}#1]}\xspace}

\NewDocumentEnvironment{revisionblock}{}{\begingroup\color{blue}\hypersetup{linkcolor=blue,citecolor=blue,urlcolor=blue}\tcbset{coltext=blue,coltitle=blue}}{\endgroup}
\graphicspath{{images/}}

\makeatletter
\AtBeginDocument{\@ifpackageloaded{cleveref}{\crefname{theorem}{Theorem}{Theorems}
    \Crefname{theorem}{Theorem}{Theorems}
    \crefname{lemma}{Lemma}{Lemmas}
    \Crefname{lemma}{Lemma}{Lemmas}
    \crefname{proposition}{Proposition}{Propositions}
    \Crefname{proposition}{Proposition}{Propositions}
    \crefname{corollary}{Corollary}{Corollaries}
    \Crefname{corollary}{Corollary}{Corollaries}
    \crefname{assumption}{Assumption}{Assumptions}
    \Crefname{assumption}{Assumption}{Assumptions}
    \crefname{definition}{Definition}{Definitions}
    \Crefname{definition}{Definition}{Definitions}
    \crefname{remark}{Remark}{Remarks}
    \Crefname{remark}{Remark}{Remarks}
    \crefname{example}{Example}{Examples}
    \Crefname{example}{Example}{Examples}
  }{}}
\makeatother

\title{(Approximate) Symmetric Extension Complexity of the \\ Spanning Tree Polytope}
\author{Sebastian Pokutta\\
Institute of Mathematics, Technische Universit\"at Berlin and\\
Zuse Institute Berlin, Germany\\
\texttt{pokutta@zib.de}}
\date{October 12, 2026}

\begin{document}
\maketitle

\begin{abstract}
We prove a tight lower bound on symmetric LP formulations for the spanning
tree polytope of the complete graph.  More precisely, let $P_{\ST}(K_n)$ be the spanning tree
polytope of $K_n$.  We show that, for all $n\ge13$, every symmetric extended formulation for $P_{\ST}(K_n)$ has at least $\binom n3$ inequalities.  Since the classical Martin formulation has a symmetric formulation
of size $O(n^3)$, this gives
\[
  \xcs(P_{\ST}(K_n))=\Theta(n^3).
\]
We also study outer relaxations of the spanning tree polytope of $K_n$ for maximization with nonnegative edge weights. We show that for every fixed approximation factor below $3/2$, symmetric LP
formulations still require $\Theta(n^3)$ inequalities for all sufficiently
large $n$, whereas factor $3/2$ admits a formulation of size $O(n^2)$.
Moreover, we provide an exact symmetric SDP formulation of size
$O(n^2)$, measured by PSD matrix order, which implies $\xc_{\mathrm{psd}}(P_{\ST}(K_n))= \Theta(n^2)$. Thus symmetric LP and SDP sizes
already differ for exact formulations, and for every fixed approximation
factor $1\le\beta<3/2$ for all sufficiently large $n$.
\end{abstract}

\section{Introduction}
The spanning tree polytope is one of the most well-known and studied objects in
polyhedral combinatorics.  For the complete graph $K_n$, let
\[
  P_{\ST}(K_n)=\conv\{\chi^T\in\{0,1\}^{E(K_n)}:T\text{ is a spanning tree of }K_n\}.
\]
\citet{Edmonds1970} gave the standard complete linear description of this polytope:
the equation $x(E(K_n))=n-1$, the nonnegativity constraints, and the subtour
inequalities $x(E(S))\le |S|-1$ for nonempty proper subsets $S\subset[n]$.
Later, \citet{Martin1991} gave a polynomial-size extended formulation of size
$O(n^3)$ for $K_n$ (more generally, $O(nm)$ for an $n$-vertex, $m$-edge graph)
using his general separation-to-formulation framework; an $O(n^3)$ formulation
for $K_n$ can alternatively be extracted from the multicommodity-flow technique
that \citet{Wong1980} introduced for the traveling salesman problem, so that the
formulation is often jointly attributed to Wong and Martin
(e.g., \citealp{KhoshkhahTheis2018}).  To the best of the
author's knowledge, no LP formulation of size $o(n^3)$ is known for $K_n$, while the
best general lower bound arises from the rank or dimension bound $\Omega(n^2)$
via Yannakakis' factorization theorem \citep{Yannakakis1991}.  Therefore
\[
  \Omega(n^2)\le \xc(P_{\ST}(K_n))\le O(n^3).
\]
Closing this gap for the
spanning tree polytope is a notorious open problem (see e.g., \citet{KhoshkhahTheis2018}).  For general background on
extended formulations, we refer the interested reader to
\citep{Goemans2015,ConfortiCornuejolsZambelli2010,Kaibel2011}.

In this paper, we study the \emph{symmetric extension complexity} of $P_{\ST}(K_n)$. The group $S_n$ acts on $K_n$ by relabeling vertices, and hence acts on $P_{\ST}(K_n)$. Many natural extended formulations respect this action; this is what we call a \emph{symmetric extended formulation}.  In particular, as, e.g., pointed out in \citet{KaibelPashkovichTheis2011}, Martin's $O(n^3)$ formulation for the spanning tree polytope is symmetric under vertex relabeling. Thus the question whether symmetry itself already forces a size of $\Omega(n^3)$ naturally arises. We answer this question in the affirmative:

\begin{theorem}[Main theorem {\leanassuming[; lower bound]}]\label{thm:main}
For all $n\ge13$,
\[
  \xcs(P_{\ST}(K_n))\ge \binom n3.
\]
Consequently,
\[
  \xcs(P_{\ST}(K_n))=\Theta(n^3).
\]
\end{theorem}

The badges \leanverified and \leanassuming indicate Lean verification;
the latter assumes the $k=3$, $n\ge13$ instance of \Cref{thm:small-index-k}.
A \emph{lower bound} qualifier limits coverage to that part of a statement.
For asymptotic statements the badges refer to the underlying finite bounds;
the cited cubic exact upper bound is not formalized.

Here $\xcs(\cdot)$ denotes the symmetric extension complexity, i.e., the minimum size of a symmetric extended formulation, with respect to the vertex-relabeling action. Note that this lower bound (and also the later approximations) in particular includes exact formulations obtained via symmetry-preserving LP hierarchies, such as the Sherali--Adams hierarchy \citep{sherali1990hierarchy}.

The proof follows the philosophy of Yannakakis' symmetric lower bounds for
matching polytopes \citep{Yannakakis1991} as well as \citet{KaibelPashkovichTheis2011}. We
use the algebraic framework of \citet{BraunPokutta2012} for symmetric extended formulations, which is basically a streamlined, more readily applicable version of \citet{KaibelPashkovichTheis2011}: it reduces the lower bound to constructing a signed
affine combination of vertices that is nonnegative on every orbit forced by the
stabilizers of a small symmetric extension, but violates a valid inequality of the
polytope.

In a nutshell, the construction will work as follows. We construct an explicit certificate on six vertices, by splitting into a triangle $S=\{1,2,3\}$ and an auxiliary triple $B=\{4,5,6\}$, we place weight $+1$ on three $S_3\times S_3$-orbits of spanning trees and weight $-1$ on one orbit.  This base gadget has positive total weight, nonnegative sums on the required stabilizer orbits, and negative slack for the triangle subtour inequality $x(E(S))\le 2$. After lifting and normalization, orbit-sum nonnegativity forces the resulting point into the projection of any hypothetical small symmetric extended formulation, while at the same time the violation of the triangle inequality forces it outside the spanning tree polytope. We then lift that certificate to $n$ vertices by ranging over possible $B$ triples disjoint from $S$ and attaching all remaining vertices as leaves to the six-vertex gadget. The key point is that the orbit-sum nonnegativity established for the local certificate survives this construction via a simple attachment-bijection argument.

\paragraph{Relation to the asymmetric problem.}
\Cref{thm:main} does not resolve the (ordinary) extension complexity of
$P_{\ST}(K_n)$, and we also did not find an easy way to \lq{}asymmetrize\rq{} the gadget.  It is very much possible in fact that the spanning tree polytope admits a smaller asymmetric formulation; this phenomenon is well known and one example is the permutahedron \citep{KaibelPashkovichTheis2011,Goemans2015}, where the extension complexity is $\Theta(n \log n)$ via AKS sorting networks, however the \emph{symmetric} extension complexity is $\Theta(n^2)$ via projection of the Birkhoff polytope.

As such, the (possibly asymmetric) extension complexity of $P_{\ST}(K_n)$ remains
open between $\Omega(n^2)$ and $O(n^3)$; any improvement
to $o(n^3)$ inequalities for exact LP formulations must however break the natural vertex-relabeling symmetry. The same conclusion holds for every fixed approximation factor $1\le\beta<3/2$ for nonnegative-weight maximization, for all sufficiently large $n$.

\paragraph{Approximate LP and exact SDP formulations.}
In \cref{sec:approximation,ssec:exact-sdp} we study symmetric LP
approximations and exact SDP formulations, respectively, for the spanning tree
polytope of $K_n$. We note that all approximation guarantees we consider here are for maximization with nonnegative edge weights, not approximation of the minimum spanning tree problem; this is basically in line with prior work on approximate extended formulations, see e.g., \citep{bfps2012,bfps2012jour,BPZ2014jour,BPR2015jour}. Confining to a subset of objective vectors makes the lower bounds stronger, however the upper bounds, i.e., the formulations we construct are relaxations of the spanning tree polytope that might not provide any guarantee for the neglected objective vectors. The exact SDP formulations have no restriction on the linear objective. We establish the following LP and SDP results:

\begin{enumerate}
\item The six-vertex construction used for \Cref{thm:main} already provides a lower bound for approximation of the following form:
      for $n\ge13$, a symmetric LP relaxation with approximation factor strictly below $5/4$ needs at least $\binom n3$ inequalities. This basically follows immediately after re-examining the proof of \Cref{thm:main}, however this factor is not tight.
 \item We then provide a quadratic-size symmetric LP that attains factor $3/2$ and a more elaborate signed
      combination supported on Hamilton paths that shows that for $n\ge13$, fewer than
      $\binom n3$ facets force an inapproximability factor of at least $\frac32(1-1/n)$.
      This establishes that $3/2$ is the sharp asymptotic threshold for symmetric LPs
      of subcubic size.
 \item We give an exact symmetric SDP formulation with one PSD
      block of order $3\binom n2$, where vertex relabeling permutes its rows
      and columns. Compression reduces the order to $3\binom n2-n$ while
      preserving symmetry by orthogonal congruences. Without requiring
      symmetry, we reduce the order further to $2\binom n2-1$. In particular, together with the lower bound in \citet{GrandeSanyal2017} we obtain for $n \geq 4$ with $m = \binom n2$ the number of edges in $K_n$
\[
 m+1\le\xc_{\mathrm{psd}}(P_{\ST}(K_n))\le2m-1.
\]
      Thus exact symmetric SDP size is quadratic, whereas exact symmetric
      LP size is cubic; the same comparison holds for every fixed
      $1\le\beta<3/2$ for all sufficiently large $n$.
\end{enumerate}

\subsection*{Related work}

We provide a brief overview of extended formulations in general and the spanning tree question in particular.

\paragraph{Extended formulations and formulation complexity.}
The modern theory of extended formulations was initiated by \citet{Yannakakis1991},
who related extension complexity to nonnegative factorizations of slack matrices
and proved the first strong limitations for symmetric formulations.  For general
background on compact extended formulations in combinatorial optimization we refer
to the surveys of \citet{ConfortiCornuejolsZambelli2010} and \citet{Kaibel2011}.
A central theme in the area is that lower bounds and reductions can often be
transported between optimization problems as shown, e.g., in \citet{bfps2012,bfps2012jour,chan2013approximate,chan2016approximate,lee2014power,lee2015lower}. In particular, \citet{BPZ2014,BPZ2014jour} developed affine reductions for LPs and SDPs, and
\citet{BPR2015,BPR2015jour} strengthened these reduction mechanisms for extended
formulations, allowing to reuse many hardness reductions from complexity theory. These reduction frameworks are compatible with symmetry when the
underlying maps respect the group actions.

The exponential lower bound for the traveling-salesman and related
polytopes was obtained by \citet{fmptw2011,fmptw2011jour}, which was subsequently simplified by \citet{kaibel2015short} and made robust to perturbations \citep{braverman2013information,BP2013commInfo,BP2013commInfoJour}.  Later,
\citet{rothvoss2014matching,rothvoss2017matchingjour} proved that the matching polytope also has exponential
extension complexity without any symmetry assumption, while
\citet{BP2014matching,BP2014matchingJour} proved that the matching polytope does
not admit fully-polynomial size relaxation schemes, which was subsequently strengthened by \citet{Sinha2018ApproxMatching}. These results show that, for
some classical polytopes, the ordinary asymmetric extension complexity is already
large. The spanning-tree polytope is different: for LP
formulations, the best known general lower bound remains quadratic, while
the best known upper bound for $K_n$ is cubic.

\paragraph{Extended formulations for spanning trees.}
Edmonds' description of the spanning tree polytope is a foundational result in
matroid polyhedral theory \citep{Edmonds1970}.  \citet{Martin1991} gave a compact
extended formulation with $O(nm)$ inequalities for an $n$-vertex graph with $m$
edges; for $K_n$ this is $O(n^3)$, a bound also obtainable from the
multicommodity-flow formulation of \citet{Wong1980}.  The general
LP extension-complexity gap for $K_n$ remains open.  Khoshkhah--Theis showed that
support-only lower-bound methods such as rectangle covers and fooling sets are
stuck near $n^2$ for this polytope \citep{KhoshkhahTheis2018,KhoshkhahTheis2017fooling}.
Smaller formulations are possible for structured graph classes:
\citet{FioriniHuynhJoretPashkovich2017} gave an $O(g^{1/2}n^{3/2}+g^{3/2}n^{1/2})$
formulation for graphs of genus $g\ge1$, and
\citet{AprileFioriniHuynhJoretWood2021} showed that for every proper minor-closed
class one obtains size $O(n^{3/2})$, and more generally
$O(n^{1+\beta})$ for graph classes closed under induced subgraphs
and having balanced separators of size $O(n^\beta)$, for a fixed
$0<\beta<1$.
For the related independence polytopes of regular matroids (which contain the
graphic case), \citet{AprileFiorini2022} proved a polynomial bound.
These results do not improve the classical cubic upper bound for
the complete-graph family $K_n$, but they reinforce that
the cubic upper bound is not an inherent feature of all spanning-tree polytopes.
In a different direction, two \emph{restricted-model} lower bounds are known for
the spanning tree polytope of $K_n$ itself.  \citet{KaibelWalter2015} proved
that its \emph{simple} extension complexity (the smallest number of facets of a
simple non-degenerate polytope projecting onto it) is
$\Omega(2^{n-o(n)})$, and
\citet{Weltge2015} and \citet{KaibelWeltge2013} proved exponential lower bounds in
the model of relaxation complexity without auxiliary variables.  Neither model
restricts to symmetry, and neither implies a lower bound for general extended
formulations with auxiliary variables; our result restricts the model in the
orthogonal direction of symmetry.

\paragraph{Symmetric extended formulations.}
\citet{Yannakakis1991} proved that symmetric extended formulations for matching
polytopes are large.  \citet{KaibelPashkovichTheis2011} sharpened and extended
the method, showing that symmetry can create superpolynomial gaps between
symmetric and asymmetric formulation sizes; the detailed development is given in
the thesis of \citet{Pashkovich2012}.  Notably, in all of this line of work the
spanning tree polytope appears only as the guiding \emph{example of a polytope
that admits a small symmetric formulation} (Martin's); the matching symmetric
lower bound we prove here was not established.  \citet{BraunPokutta2012} later gave
an algebraic approach that packages the relevant orbit-sum criterion in a
convenient form.  We use that formulation
below.  The required small-index group-theoretic result we use is standard; we cite it in
the form used by \citet{BraunPokutta2012} and by the symmetric-SDP work of
\citet{BraunEtAl2017SymmetricSDP}, which derives it from Theorems 5.2A and 5.2B
of \citet{DixonMortimer1996}.

  \paragraph{Semidefinite formulations.} We follow the standard model of (symmetric) semidefinite extended formulations and (equivariant) conic lifts, which are characterized by cone
  factorizations \citep{GouveiaParriloThomas2013}; \citet{BraunEtAl2017SymmetricSDP} prove lower bounds for coordinate-symmetric SDPs for matching by relating them to sum-of-squares representations. Our spanning tree constructions give
quadratic-size exact symmetric SDP lifts, in contrast to the
cubic symmetric LP bound.
For exact formulations, \citet[Theorems~1.2 and~4.9]{GrandeSanyal2017} imply that the semidefinite extension complexity of $P_{\ST}(K_n)$ is at least $\binom n2+1$ for $n\ge4$. We also prove a similar quadratic lower bound $\binom{n-1}{2}$ for PSD lifts of outer relaxations $P_{\ST}(K_n)\subseteq R\subseteq\R_+^{E(K_n)}$, without any symmetry assumption (see \Cref{prop:nonnegative-psd-order}).

\section{Preliminaries}

We will now briefly recall necessary notions and notation.

\subsection{The spanning tree polytope}

Let $K_n=([n],E_n)$ be the complete graph.  For $S\subseteq[n]$, we write
$E(S)$ for the set of edges with both endpoints in $S$.  Edmonds' description of
the spanning tree polytope is given by
\begin{align}
  P_{\ST}(K_n)=\{x\in\R_+^{E_n}:\ &x(E_n)=n-1,\nonumber\\
  &x(E(S))\le |S|-1\quad\text{for all }S\subseteq[n],\ 2\le |S|\le n-1\}.
\end{align}
For a triangle $S$ with $|S|=3$, the subtour inequality becomes
\(x(E(S))\le |S|-1=2\) and is called a \emph{triangle subtour inequality}. Given a spanning tree $T$, its
slack (relative to $S$) is given by
\[
  h_S(T) = 2-|T\cap E(S)|,
\]
and possible values of $h_S(T)$ are $0,1,2$ depending on whether the induced forest on
$S$ has two, one, or zero edges.

\subsection{Symmetric extended formulations}

We recall the standard terminology for symmetric extensions, following
\citet{Yannakakis1991}, \citet{KaibelPashkovichTheis2011}, and
\citet{BraunPokutta2012}.

Let a finite group $G$ act affinely on $\R^m$.  A polytope $P\subseteq\R^m$ is a
\emph{$G$-polytope} if
\[
  gP=P\qquad\text{for every }g\in G.
\]
Equivalently, the action permutes the vertices and the face lattice of $P$.
In this paper $G$ will be the full symmetric group $S_n$, acting on $K_n$ by
relabeling vertices.  We will also use alternating subgroups of the form
$A([n]\setminus W)$ inside stabilizers.  The vertex-relabeling action naturally induces
actions on edges, spanning trees, and the spanning tree polytope.

An \emph{extension} of a polytope $P\subseteq\R^m$ is a polyhedron
$Q\subseteq\R^d$ together with an affine projection $p:\R^d\to\R^m$ such that
$p(Q)=P$.  Its size is the number of facet-defining inequalities of $Q$; equations
are not counted, see e.g., \citet{BPZ2014,BPR2015,BPZ2014jour,BPR2015jour} for some background on this complexity model.  If $P$ is a $G$-polytope, a \emph{$G$-symmetric extension} is an extension $p:Q\to P$ together with an affine action of $G$ on $Q$ such that
\[
  p(gq)=g\,p(q)
  \qquad\text{for every }g\in G\text{ and }q\in Q.
\]
Note that the usual coordinate-symmetric formulation, in which $G$ permutes the auxiliary
variables or facets, is equivalent for lower-bound purposes to this invariant
formulation after passing to the slack/subspace representation; see
\citet{KaibelPashkovichTheis2011,BraunPokutta2012}.

For a $G$-polytope $P$, define its \emph{symmetric extension complexity ($\xcs(\cdot)$)} via
\[
  \xcs_G(P) \doteq \min\{\text{number of inequalities in a }G\text{-symmetric extension of }P\},
\]
and for the spanning tree polytope we abbreviate
\[
  \xcs(P_{\ST}(K_n)) \doteq \xcs_{S_n}(P_{\ST}(K_n)).
\]

For a finite set $Y$, write $\Sym(Y)$ for the full symmetric group on $Y$.
Thus, if $|Y|=3$, then $\Sym(Y)$ is naturally isomorphic to the usual group $S_3$.
For $W\subseteq[n]$, the notation $A([n]\setminus W)$ denotes the \emph{alternating
group on the complement of $W$}, viewed as the subgroup of $A_n$ that fixes every
vertex of $W$ and acts by even permutations on $[n]\setminus W$. We will establish our lower bound for $S_n$-symmetric extensions and the alternating groups
$A([n]\setminus W)$ will appear naturally through the small-index theorem for stabilizers of facets in a small $S_n$-symmetric formulation as we will see below. To this end, for a group action on a set $X$, we write for the \emph{orbit} and \emph{stabilizer}, respectively:
\[
  \operatorname{Orb}_G(x)=\{gx:g\in G\},
  \qquad
  \operatorname{Stab}_G(x)=\{g\in G:gx=x\}.
\]
We will also use the well-known \emph{orbit-stabilizer identity}
\[
  |\operatorname{Orb}_G(x)|=[G:\operatorname{Stab}_G(x)].
\]

Finally, we recall the upper bound result of \citet{Martin1991}.

\begin{proposition}[Symmetric extension upper bound]\label{prop:symmetric-upper}
The spanning tree polytope $P_{\ST}(K_n)$ has an $S_n$-symmetric extended
formulation of size $O(n^3)$.
\end{proposition}

\subsection{The Orbit criterion}

We use the orbit criterion of \citet{BraunPokutta2012} as a black-box algebraic form of the symmetric-extension framework from \citet{Yannakakis1991,KaibelPashkovichTheis2011}.

\begin{theorem}[Orbit criterion; \citep{BraunPokutta2012}]\label{thm:bp-general}
Let a finite group $G$ act affinely on a polytope $P=\conv(X)$.  Let $Q$ be a
$G$-symmetric extension of $P$.  For every facet $j$ of $Q$, let $G_j\le G$ be
its stabilizer.  The $G_j$-orbit partition of $X$ is the partition whose parts are
sets of the form $\operatorname{Orb}_{G_j}(x)$.  Let $\mathcal F_j$ be any
refinement of this orbit partition.  If $c=(c_x)_{x\in X}\in\R^X$ satisfies
\[
  \sum_{x\in X}c_x=1
\]
and
\[
  \sum_{x\in F}c_x\ge0
  \qquad\text{for every }F\in\mathcal F_j\text{ and every facet }j,
\]
then
\[
  \sum_{x\in X}c_x x\in P.
\]
\end{theorem}

The group-theoretic property needed to specialize \Cref{thm:bp-general} is the
following small-index theorem, which follows directly from \citep[Theorems~5.2A and~5.2B]{DixonMortimer1996}; see also \citep{KaibelPashkovichTheis2011}.

\begin{theorem}[Small-index theorem]\label{thm:small-index-k}
Let $1\le k<n/4$, let $n\ge10$, and let $G\le S_n$.  If
\[
  [S_n:G]<\binom nk,
\]
then there is a subset $W\subseteq[n]$ with $|W|<k$ such that
\[
  A([n]\setminus W)\le G.
\]
\end{theorem}

The range $k<n/4$ is a conservative form of the standard small-index theorem
and is sufficient for our purposes.  It is not sharp for the case used here:
for $k=3$ the conclusion already holds for $n\ge10$.  The only obstruction to
a naive extension as far as $k\le n/2$ is, up to conjugacy, the imprimitive
subgroup preserving a bipartition of $[n]$ into two equal blocks, present only
for even $n$; its index is $\frac12\binom{n}{n/2}$, which is at least
$\binom n3$ for even $n\ge10$.  Thus this exception cannot occur under our
strict index hypothesis $[S_n:G]<\binom n3$.  We keep $n\ge13$ throughout
only because it is the clean threshold obtained from the quoted $k<n/4$
version with $k=3$.

\begin{corollary}[Small-index orbit criterion]\label{cor:bp-k}
Let $n\ge10$ and $1\le k<n/4$, and let $P=\conv(X)$ be an $S_n$-polytope.  Suppose $P$ has an
$S_n$-symmetric extended formulation with fewer than $\binom nk$ facets.  If
$c=(c_x)_{x\in X}\in\R^X$ satisfies
\[
  \sum_{x\in X}c_x=1
\]
and, for every $W\subseteq[n]$ with $|W|<k$ and every $A([n]\setminus W)$-orbit
$\mathcal O$ on $X$,
\[
  \sum_{x\in\mathcal O}c_x\ge0,
\]
then
\[
  \sum_{x\in X}c_x x\in P.
\]
\end{corollary}

\begin{proof}
Let $j$ be a facet of a purported symmetric extension.  Its orbit under $S_n$ has
cardinality $[S_n:G_j]$, where $G_j$ is the facet's stabilizer, and this orbit has size
less than $\binom nk$ because the extension has fewer than $\binom nk$ facets.
By \Cref{thm:small-index-k}, $G_j$ contains $A([n]\setminus W_j)$ for some
$|W_j|<k$.  Therefore every $G_j$-orbit on $X$ is a union of
$A([n]\setminus W_j)$-orbits.  Nonnegativity on all such smaller orbits implies
nonnegativity on the $G_j$-orbits.  Applying \Cref{thm:bp-general} gives
the claim.
\end{proof}

In particular, taking $k=3$, if $n\ge13$ and $P$ has an $S_n$-symmetric extended
formulation with fewer than $\binom n3$ facets, then it suffices to check
nonnegativity on the orbits of the subgroups $A([n]\setminus W)$ with $|W|\le2$.
Thus the proof of \Cref{thm:main} reduces to constructing a vector of signed weights on spanning trees with positive total weight (after normalizing this is effectively an affine combination), nonnegative sums on all
$A([n]\setminus W)$-orbits for $|W|\le2$, and negative slack for one valid
inequality. We refer to such weights as \emph{signed certificate}.

\section{Building the affine combination}

The construction proceeds in two steps.  First we build a local signed weight vector on the spanning trees of $K_6$.  This six-vertex object is only a certificate gadget: it is not meant to prove the lower bound for $K_6$ itself.  Second, for arbitrary $n$ we lift this local certificate to $K_n$ by choosing an auxiliary triple and attaching all remaining vertices as leaves; \Cref{fig:certificate-and-lift} shows both the base certificate and this lift.  The condition $n\ge13$ enters only when we invoke the small-index theorem in \Cref{cor:bp-k} with $k=3$.

\begin{figure}[t]
  \centering
  \setlength{\fboxsep}{2pt}
  \setlength{\fboxrule}{0.5pt}
  \includegraphics[width=\dimexpr\linewidth-2\fboxsep-2\fboxrule\relax]{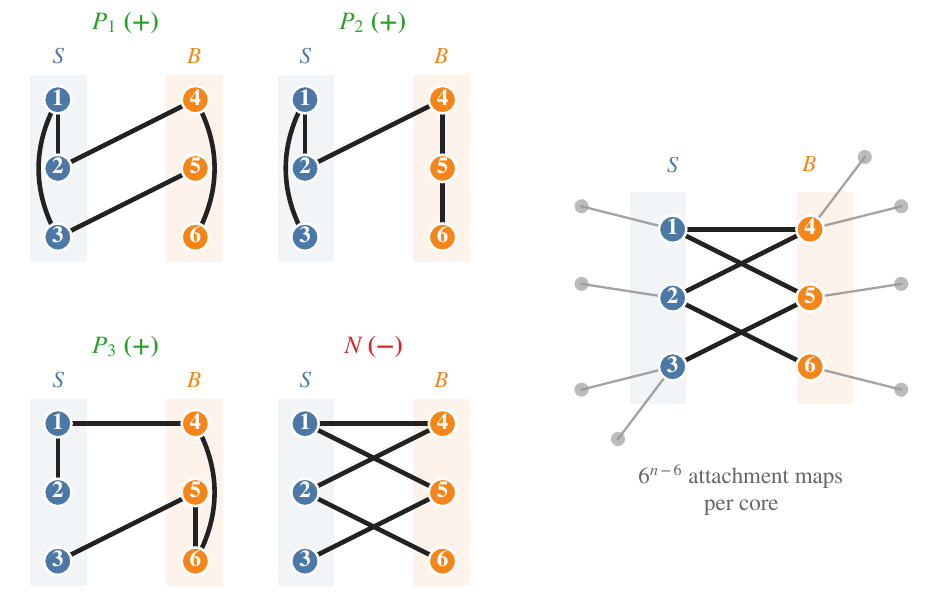}
  \caption{Left: the base $K_6$ signed certificate, with the four
  representatives arranged in a two-by-two grid. All five six-vertex cores
  are drawn at the same scale. The vertices are split into
  the triangle $S=\{1,2,3\}$ and the auxiliary triple $B=\{4,5,6\}$. We put
  weight $+1$ on the $S_3\times S_3$-orbits of $P_1,P_2,P_3$ and weight $-1$ on
  the orbit of $N$. Each orbit has size $36$. Right: the lift to $n$ vertices,
  obtained by choosing $B$ and attaching every remaining vertex as a leaf to the
  six-vertex core.}
  \label{fig:certificate-and-lift}
\end{figure}

\subsection{Affine combination on six vertices}

Fix two triples
\[
  S=\{1,2,3\},\qquad B=\{4,5,6\}.
\]
The two triples play symmetric roles in the base certificate.  We distinguish
$S$ only because the eventual violated inequality is the triangle subtour
inequality on $S$; in the lift, $B$ is the second triple chosen from the remaining
vertices.  Let $\Gamma=\Sym(S)\times\Sym(B)$ be the group that permutes $S$ and
$B$ separately.
Consider the four spanning trees on $S\cup B$:
\[
\begin{aligned}
P_1&=\{12,13,24,35,46\},\\
P_2&=\{12,13,24,45,56\},\\
P_3&=\{12,14,35,46,56\},\\
N&=\{14,15,24,26,35\}.
\end{aligned}
\]
All four graphs are paths on the six vertices, with the two triples
$S$ and $B$ interleaved along the path; see \cref{fig:certificate-and-lift}.  They
belong to different $\Gamma$-orbits: simply consider the
invariant pair
\[
  \bigl(|T\cap E(S)|,\ |T\cap \delta(S)|\bigr),
\]
where $\delta(S)$ is the set of edges with exactly one endpoint in $S$: these
pairs are $(2,2)$, $(2,1)$, $(1,2)$, and $(0,5)$ for $P_1,P_2,P_3$, and $N$,
respectively.
We put weight $+1$ on every tree in the orbits $\Gamma P_1$, $\Gamma P_2$, and
$\Gamma P_3$, weight $-1$ on every tree in $\Gamma N$, and weight zero on all
other six-vertex trees.  Denote the resulting signed weights by $\lambda^{(6)}$.

\begin{lemma}[Base total and slack \leanverified]\label{lem:base-total-slack}
The base certificate satisfies
\[
  \sum_T \lambda_T^{(6)}=72
\]
and
\[
  \sum_T\lambda_T^{(6)}\bigl(2-|T\cap E(S)|\bigr)=-36.
\]
\end{lemma}

\begin{proof}
The action of $\Gamma=\Sym(S)\times\Sym(B)$ on each of the four representatives
$P_1,P_2,P_3,N$ is free: no nonidentity permutation preserving the two triples
fixes any one of these labeled paths.  Hence each orbit has size
$|\Gamma|=|\Sym(S)\times\Sym(B)|=(3!)^2=36$.  Thus the total signed weight is
$3\cdot36-36=72$.  The representatives $P_1$ and $P_2$ have two edges inside
$S$, the representative $P_3$ has one edge inside $S$, and $N$ has no edge inside
$S$.  Therefore the signed sum of the triangle-slack function
$h_S(T)=2-|T\cap E(S)|$ over the four weighted orbits is
\[
  72(2-2)+36(2-1)-36(2-0)=-36.
\]
\end{proof}

Recall that the slack function is affine-linear in the characteristic vector $\chi$ of a given tree $T$, i.e., $2-|T\cap E(S)|=2-\langle \chi^T,\mathbf 1_{E(S)}\rangle$.  Therefore, for any signed affine combination $y=\sum_T c_T\chi^T$ with $\sum_Tc_T=1$, the same
signed combination of triangle slacks equals $2-y(E(S))$, which is why a negative
signed slack sum later certifies violation of the triangle subtour inequality.

\begin{lemma}[Base orbit sums \leanverified]\label{lem:base-orbit-sums}
For every $W\subseteq S\cup B$ with $|W|\le2$ and every orbit $\mathcal O$ of
$A((S\cup B)\setminus W)$ on six-vertex trees,
\[
  \sum_{T\in\mathcal O}\lambda_T^{(6)}\ge0.
\]
\end{lemma}

\begin{proof}
Let
\[
  C_1=\Gamma P_1,\qquad C_2=\Gamma P_2,
  \qquad C_3=\Gamma P_3,
  \qquad C_4=\Gamma N.
\]
For an $A((S\cup B)\setminus W)$-orbit $\mathcal O$, record the intersection
vector
\[
  v(\mathcal O)=
  (|\mathcal O\cap C_1|,|\mathcal O\cap C_2|,
    |\mathcal O\cap C_3|,|\mathcal O\cap C_4|).
\]
The signed sum on $\mathcal O$ is $v_1+v_2+v_3-v_4$.
An orbit disjoint from $C_1\cup C_2\cup C_3\cup C_4$ has
$v(\mathcal O)=(0,0,0,0)$ and signed sum zero. We therefore only need to
classify the orbits meeting this weighted support. By the symmetry exchanging
$S$ and $B$, and by the transitivity of $\Gamma$ on each triple, it suffices to
consider the four cases
\[
\begin{gathered}
  W=\varnothing,
  \qquad |W|=1,
  \qquad W\subseteq S\text{ or }W\subseteq B,\ |W|=2,\\
  \text{and}\qquad |W\cap S|=|W\cap B|=1.
\end{gathered}
\]
For each case, the acting group is the alternating group on the unfixed vertices.
For each of the four types of $W$, the fixed vertices determine finitely many
incidence patterns of the path.  For each such pattern, choose one representative
$T$, compute its stabilizer in $A((S\cup B)\setminus W)$, and use the
orbit-stabilizer identity to determine the size of its orbit.  Intersecting that
orbit with the four $\Gamma$-orbits $C_1,C_2,C_3,C_4$ gives
the exhaustive list of nonzero intersection vectors in
\Cref{tab:base-orbit-sums}.
For example, when $W\subseteq S$ has two vertices, the vector $(0,3,3,0)$ occurs
when the two fixed vertices are adjacent in the path, while $(3,0,0,3)$ occurs
when they are not adjacent but have the same incidence pattern with the moving
vertices; the other rows are obtained similarly by the position and degree pattern
of the fixed vertices along the six-vertex path.  In every listed case, the
signed sum shown in \Cref{tab:base-orbit-sums} is nonnegative.  This proves
the claim.
\end{proof}

\begin{table}[t]
\small
\centering
\begin{tabularx}{\textwidth}{@{}lX@{}}
\toprule
Type of $W$ & Nonzero intersection vectors $(p_1,p_2,p_3,n)$ and signed sums
$p_1+p_2+p_3-n$\\
\midrule
$|W|=0$
& $(36,36,36,36;72)$\\[0.25em]
$|W|=1$
& $(0,6,12,6;12)$, $(6,6,6,6;12)$, $(12,6,0,6;12)$\\[0.25em]
$|W|=2$, $W\subseteq S$ or $W\subseteq B$
& $(0,0,3,3;0)$, $(0,0,6,0;6)$, $(0,3,0,3;0)$, $(0,3,3,0;6)$,
  $(3,0,0,3;0)$, $(3,3,0,0;6)$, $(6,0,0,0;6)$\\[0.25em]
$|W|=2$, $|W\cap S|=|W\cap B|=1$
& $(0,0,2,2;0)$, $(0,0,4,0;4)$, $(0,2,0,2;0)$, $(0,2,2,0;4)$,
  $(0,2,4,2;4)$, $(2,0,0,2;0)$, $(2,2,0,0;4)$, $(2,2,2,2;4)$,
  $(4,0,0,0;4)$, $(4,2,0,2;4)$\\
\bottomrule
\end{tabularx}
\caption{The finite base-orbit classification for \Cref{lem:base-orbit-sums}.
Here $(p_1,p_2,p_3,n)$ abbreviates
$(|\mathcal O\cap\Gamma P_1|,|\mathcal O\cap\Gamma P_2|,
|\mathcal O\cap\Gamma P_3|,|\mathcal O\cap\Gamma N|)$, and the number after the
semicolon is the signed sum $p_1+p_2+p_3-n$.
Orbits disjoint from the weighted support have the additional
vector $(0,0,0,0)$ and signed sum zero, omitted from the table.}
\label{tab:base-orbit-sums}
\end{table}

\subsection{Lifting the certificate and preserved orbit-sum nonnegativity}

Now fix $n\ge6$ and keep $S=\{1,2,3\}$.  For every auxiliary triple
$B\subseteq[n]\setminus S$ with $|B|=3$, identify $S\cup B$ with the six vertices
above.  For every spanning tree $F$ on the six-vertex core $S\cup B$ and every
attachment map
\[
  a:[n]\setminus(S\cup B)\to S\cup B,
\]
we define its lift as
\[
  \operatorname{Lift}(F,a)=F\cup\{\{c,a(c)\}:c\in[n]\setminus(S\cup B)\}.
\]
Observe that this is a spanning tree: it is obtained from the six-vertex tree
$F$ by adding all remaining vertices as leaves.  Define
\[
  \operatorname{sgn}(F)=
  \begin{cases}
    +1, & F\in \Gamma P_1\cup\Gamma P_2\cup\Gamma P_3,\\
    -1, & F\in \Gamma N,\\
    0, & \text{otherwise.}
  \end{cases}
\]
on all six-vertex trees.  Now define the lifted signed weight of a spanning tree
$T$ by
\[
  \lambda_T=
  \sum_{(B,F,a):\operatorname{Lift}(F,a)=T}\operatorname{sgn}(F).
\]

\begin{lemma}[Lifted total and slack \leanverified]\label{lem:lift-total-slack}
The lifted certificate satisfies
\[
  \sum_T\lambda_T=72\cdot 6^{n-6}\binom{n-3}{3}
\]
and
\[
  \sum_T\lambda_T\bigl(2-|T\cap E(S)|\bigr)
  =-36\cdot 6^{n-6}\binom{n-3}{3}.
\]
\end{lemma}

\begin{proof}
There are $\binom{n-3}{3}$ choices for the auxiliary triple $B$.  For each such
$B$, there are $6^{n-6}$ attachment maps to the six-vertex core $S\cup B$.  Each
copy contributes total weight $72$ and triangle slack $-36$ by
\Cref{lem:base-total-slack}.  Multiplying gives both identities.
\end{proof}

Next we show that the orbit sums remain nonnegative after the lift, which is essential for our proof later.  We prove a slightly stronger invariant-predicate statement, which avoids explicitly enumerating full $n$-vertex orbits.

\begin{lemma}[Attachment-count invariance \leanverified]\label{lem:attachment-count}
Fix $W\subseteq[n]$ with $|W|\le2$, a core $U=S\cup B$, and a predicate
$\mathrm{Good}$ on spanning trees that is invariant under $A([n]\setminus W)$.
For a tree $F$ on $U$, let
\[
  N_{U,\mathrm{Good}}(F)=
  \card{\{a:[n]\setminus U\to U:\operatorname{Lift}(F,a)\text{ satisfies }\mathrm{Good}\}}.
\]
If $\rho\in A(U\setminus(W\cap U))$, then
\[
  N_{U,\mathrm{Good}}(F)=N_{U,\mathrm{Good}}(\rho F).
\]
\end{lemma}

\begin{proof}
Extend $\rho$ by the identity outside $U$.  The extension lies in
$A([n]\setminus W)$.  The map
\[
  a\mapsto \rho\circ a
\]
is a bijection on attachment maps $[n]\setminus U\to U$.  Moreover
\[
  \operatorname{Lift}(\rho F,\rho\circ a)=\rho\operatorname{Lift}(F,a).
\]
Since $\mathrm{Good}$ is invariant under the extended permutation, the bijection
preserves membership in $\mathrm{Good}$.  Hence the two counts are equal.
\end{proof}

\begin{lemma}[Lifted invariant-predicate nonnegativity \leanverified]\label{lem:lift-orbit-sums}
Let $W\subseteq[n]$ with $|W|\le2$, and let $\mathrm{Good}$ be any predicate on
spanning trees invariant under $A([n]\setminus W)$.  Then
\[
  \sum_T\lambda_T\mathbf 1_{\mathrm{Good}}(T)\ge0.
\]
In particular, all $A([n]\setminus W)$-orbit sums of $\lambda$ are nonnegative.
\end{lemma}

\begin{proof}
Fix a core $U=S\cup B$ and write $W_U=W\cap U$.  For a six-vertex tree $F$ on
$U$, let
\[
  N_U(F)=\card{\{a:[n]\setminus U\to U:\operatorname{Lift}(F,a)\text{ satisfies }\mathrm{Good}\}}.
\]
By \Cref{lem:attachment-count}, the function $N_U(F)$ is constant on every orbit
of $A(U\setminus W_U)$ acting on the spanning trees of the core $U$.  The contribution of
this fixed core to the signed sum is
\[
  \sum_F \operatorname{sgn}(F)N_U(F),
\]
where the sum is over all spanning trees on $U$.
Partition this full set of trees into $A(U\setminus W_U)$-orbits
and group the sum accordingly. The support of $\operatorname{sgn}$ need not
itself be invariant. On each such orbit
$\mathcal O'$, the value $N_U(F)$ is constant, so the contribution of
$\mathcal O'$ equals that constant times
\[
  \sum_{F\in\mathcal O'}\operatorname{sgn}(F).
\]
Since $\card{W_U}\le2$, this inner signed sum is nonnegative by
\Cref{lem:base-orbit-sums}.  Hence the contribution of the core $U$ is
nonnegative.  Summing over all choices of the auxiliary triple $B$ preserves
nonnegativity.  Finally, taking $\mathrm{Good}$ to be membership in a fixed full
$A([n]\setminus W)$-orbit gives the orbit-sum statement.
\end{proof}

\section{\texorpdfstring{Proof of \cref{thm:main}}{Proof of the main theorem}}

We are now ready to put everything together and prove \cref{thm:main}.

\begin{proof}[Proof of \cref{thm:main}]
Suppose, for contradiction, that
$P_{\ST}(K_n)$ has an $S_n$-symmetric extended formulation with fewer than
$\binom n3$ facets.  Let $\lambda$ be the lifted signed certificate above and set
\[
  c_T=\frac{\lambda_T}{\sum_U\lambda_U}.
\]
By \Cref{lem:lift-total-slack}, the denominator is positive, so
$\sum_Tc_T=1$.  By \Cref{lem:lift-orbit-sums}, the coefficients have
nonnegative sums on every $A([n]\setminus W)$-orbit for every $|W|\le2$.
The $k=3$ case of \Cref{cor:bp-k} implies that
\[
  y=\sum_T c_T\chi^T
\]
lies in the spanning tree polytope.

However, again by \Cref{lem:lift-total-slack},
\[
  \sum_T c_T\bigl(2-|T\cap E(S)|\bigr)<0.
\]
Equivalently,
\[
  y(E(S))>2,
\]
which violates the triangle subtour inequality $x(E(S))\le2$ valid for
$P_{\ST}(K_n)$.  This contradiction proves
\[
  \xcs(P_{\ST}(K_n))\ge \binom n3
\]
for all $n\ge13$.  The matching upper bound
$\xcs(P_{\ST}(K_n))=O(n^3)$ follows from \Cref{prop:symmetric-upper},
so $\xcs(P_{\ST}(K_n))=\Theta(n^3)$.
\end{proof}

\section{Approximate LP formulations for the spanning tree polytope}
\label{sec:approximation}

The construction used to prove \cref{thm:main} already rules out approximation factors strictly below $5/4$ for symmetric LP formulations of subcubic size. We first derive this consequence without changing the construction as a warm-up. We then give a quadratic-size LP relaxation with factor
$3/2$ and strengthen the lower bound via a new construction to show that this factor is
asymptotically best possible for symmetric LPs of subcubic size.

\subsection{Approximation model and outer relaxations}
\label{ssec:approx-model}

We follow the standard approximate extended formulation model for both LP formulations and SDP formulations as established and used in \cite{bfps2012,bfps2012jour,chan2013approximate,chan2016approximate,lee2015lower,lee2014power,BPR2015,BPR2015jour,BPZ2014,BPZ2014jour}. Throughout this section, approximation refers to \emph{maximization with
nonnegative edge weights}. In the following let $E=E(K_n)$ and $m=\binom n2$.

\begin{definition}[Approximation factor]
An outer relaxation $R\supseteq P_{\ST}(K_n)$ has approximation factor
$\beta\ge1$ if
\begin{equation}\label{eq:approx-factor}
 \sup_{x\in R}c^\top x
 \le \beta\max_{T\text{ spanning tree}}c^\top\chi^T
 \qquad\text{for every }c\in\R_+^E.
\end{equation}
\end{definition}

An LP formulation of $R$ is a polyhedron $Q$ with an affine projection
$p(Q)=R$; symmetry and size are defined as in the exact case.

The orbit criterion in \Cref{thm:bp-general,cor:bp-k} has been only stated for exact extensions so far. We remedy this with the following version which allows us to use the same signed weights constructions for outer relaxations; for convenience \Cref{lem:outer-orbit} has been streamlined for its application to the spanning tree polytope; the proof is similar to the original one in \citet{BraunPokutta2012}.

\begin{lemma}[Orbit criterion for outer relaxations \leanassuming]\label{lem:outer-orbit}
Let $n\ge13$, and let $p:Q\to R$ be an $S_n$-symmetric LP formulation
with $P_{\ST}(K_n)\subseteq R$. Suppose $Q$ has fewer than $\binom n3$
facets. If signed weights $(a_T)_T$ on spanning trees satisfy
\[
 \sum_T a_T=1,\qquad
 \sum_{T\in\mathcal O}a_T\ge0
\]
for every $A([n]\setminus W)$-orbit $\mathcal O$ and every $|W|\le2$,
then $\sum_T a_T\chi^T\in R$.
\end{lemma}
\begin{proof}
Put $G=S_n$ and $L=\operatorname{aff}(Q)$. We use the given affine
$G$-action on $Q$, which also preserves $L$, and write $gT$ for tree
relabeling for $g \in G$.

\emph{Equivariant lifts of the tree vertices.}
Since $P_{\ST}(K_n)\subseteq p(Q)$, choose $q_T\in Q$ with
$p(q_T)=\chi^T$ for every spanning tree $T$, and set
\[
 s(T)=\frac1{|G|}\sum_{g\in G}g^{-1}\cdot q_{gT}.
\]
Every summand lies in $Q$, so convexity gives $s(T)\in Q$.
Affineness and equivariance of $p$ give
\[
 p(s(T))=\frac1{|G|}\sum_{g\in G}g^{-1}\cdot\chi^{gT}=\chi^T.
\]
For $h\in G$, reindexing the first sum by $u=gh$ gives
$s(hT)=h\cdot s(T)$; here the affine action commutes with an average
because its coefficients sum to one.

\emph{Invariant facet slacks.}
Let $F_1,\ldots,F_r$ be the facets of $Q$. Working in its affine hull,
choose affine functions $h_j:L\to\R$ such that
\[
 Q=\{z\in L:h_j(z)\ge0\text{ for all }j\},
 \qquad F_j=\{z\in Q:h_j(z)=0\}.
\]
Let $G_j=\{g\in G:gF_j=F_j\}$ be the setwise stabilizer of $F_j$,
and define
\[
 \bar h_j(z)=\frac1{|G_j|}\sum_{g\in G_j}h_j(g\cdot z).
\]
Each summand is nonnegative on $Q$ and defines the same facet $F_j$.
On $L$, it is a positive scalar multiple of $h_j$: the two affine
functions vanish on the same hyperplane $\operatorname{aff}(F_j)$,
and are positive on $Q\setminus F_j$.
Consequently, $\bar h_j\ge0$ defines the same halfspace in $L$ as
$h_j\ge0$. Reindexing the sum also shows that
$\bar h_j(g\cdot z)=\bar h_j(z)$ for every $g\in G_j$.

\emph{Nonnegative signed slack sums.}
Fix a facet $j$. Since $G$ permutes the $r$ facets, orbit--stabilizer gives
\[
 [G:G_j]=|G\cdot F_j|\le r<\binom n3.
\]
By \Cref{thm:small-index-k} with $k=3$, there is a set $W_j$ with
$|W_j|\le2$ such that $H_j:=A([n]\setminus W_j)\le G_j$.
Equivariance of $s$ and invariance of $\bar h_j$ imply for any $h\in H_j$
\[
 \bar h_j(s(hT))=\bar h_j(h\cdot s(T))=\bar h_j(s(T)).
\]
Thus this slack value is constant on each $H_j$-orbit of trees.
Let $\mathcal P_j$ be the set of those orbits, and choose a representative
$T_{\mathcal O}$ in each $\mathcal O\in\mathcal P_j$. Grouping the sum by
orbits gives
\[
 \sum_T a_T\bar h_j(s(T))
 =\sum_{\mathcal O\in\mathcal P_j}
   \underbrace{\bar h_j(s(T_{\mathcal O}))}_{\ge0\text{ since }s(T_{\mathcal O})\in Q}
   \underbrace{\left(\sum_{T\in\mathcal O}a_T\right)}_{\ge0\text{ by hypothesis}}
 \ge0.
\]

\emph{Feasibility and projection.}
Set $q^*=\sum_T a_Ts(T)$. The coefficients may be negative, but they sum
to one, so $q^*\in L$. Affineness of each slack then gives
\[
 \bar h_j(q^*)=\sum_T a_T\bar h_j(s(T))\ge0
 \qquad\text{for every facet }j.
\]
The equivalent facet inequalities in $L$ therefore imply $q^*\in Q$.
Finally, affineness of $p$ and $\sum_Ta_T=1$ give
\[
 p(q^*)=\sum_T a_Tp(s(T))=\sum_Ta_T\chi^T.
\]
Since $p(Q)=R$, this is the required membership in $R$.
\end{proof}

\subsection{A factor-\texorpdfstring{$5/4$}{5/4} lower bound from the original construction}
\label{ssec:old-approximation}

As a warm-up we will show that the original construction from \Cref{thm:main} proves more than exact symmetric LP complexity and, in fact, excludes approximation factors strictly below $5/4$ at subcubic
size.

\begin{corollary}[approximate LP-hardness for original construction\leanassuming]\label{cor:old-approximation}
Let $n\ge13$ and $1\le\beta<5/4$. Every $S_n$-symmetric LP outer
relaxation of $P_{\ST}(K_n)$ with factor $\beta$ has at least $\binom n3$
inequalities.
\end{corollary}
\begin{proof}
Use exactly the lifted weights $\lambda_T$ from
\cref{lem:lift-total-slack,lem:lift-orbit-sums}, and normalize them by their
positive total weight $M=\sum_T\lambda_T$. Then
\begin{equation}\label{eq:old-normalized-slack}
 \sum_T\frac{\lambda_T}{M}\bigl(2-|T\cap E(S)|\bigr)
 =-\frac12,
 \qquad S=\{1,2,3\}.
\end{equation}
If the formulation had fewer than $\binom n3$ facets,
\cref{lem:outer-orbit} would give a point
$y=\sum_T(\lambda_T/M)\chi^T\in R$ with $y(E(S))=5/2$.
The maximum triangle weight of a spanning tree is two. Taking
$c=\mathbf1_{E(S)}$ in \eqref{eq:approx-factor} yields
$5/2\le2\beta$, a contradiction.
\end{proof}

\subsection{A quadratic-size LP relaxation with factor \texorpdfstring{$3/2$}{3/2}}
\label{ssec:pseudoforest-lp}
This argument above alone neither excludes the factor $5/4$ nor does it establish its optimality. The next
construction gives a LP relaxation with an approximation factor of $3/2$ of quadratic size, raising the
question of what happens between these two constants.

\begin{proposition}[$3/2$-approximate LP formulation of quadratic size \leanverified]\label{prop:approx-lp-upper}
For $n\ge3$, there is an $S_n$-symmetric LP outer relaxation of
$P_{\ST}(K_n)$ with $3m+n$ inequalities and approximation factor exactly
$3/2$. After imposing $x(E)=n-1$, its factor is still at most $3/2$ and
is exactly $3/2$ for $n\ge4$.
\end{proposition}
\begin{proof}
We introduce variables $y_{uv}$ for ordered pairs of distinct vertices and consider the following LP of quadratic size:
\begin{align}
 y_{uv}&\ge0 && u\ne v,\label{eq:pseudo-nonnegative}\\
 y_{uv}+y_{vu}&\le1 && \{u,v\}\in E,\label{eq:pseudo-edge}\\
 \sum_{v\ne u}y_{uv}&\le1 && u\in[n],\label{eq:pseudo-degree}\\
 x_{uv}&=y_{uv}+y_{vu} && \{u,v\}\in E.\label{eq:pseudo-projection}
\end{align}
Every spanning tree has a feasible orientation towards a root, so this
is an outer relaxation. The inequalities describe a bipartite matching
polytope: the left vertices are the vertices $u$ of $K_n$, the right
vertices are its undirected edges, and $y_{uv}$ is the variable for the
incidence of $u$ with $\{u,v\}$. Thus the polytope in $y$ is integral.

Each integral solution describes an oriented simple graph with outdegree
at most one. Every connected component consequently has at most one
cycle, i.e., it is a pseudoforest. Delete a minimum-weight
edge from each cycle. Each cycle has at least three edges, so the remaining
forest retains at least two thirds of the nonnegative objective value.
Extend the forest to a spanning tree in $K_n$, without decreasing its
weight. Integrality now proves the factor $3/2$ for every feasible point.
There are $2m+m+n=3m+n$ inequalities, and relabeling vertices permutes
the variables and constraints.

A cyclically oriented triangle with unit edge weights, and zero weights
elsewhere, attains relaxation value three and tree optimum two. Adding
$x(E)=n-1$ preserves containment and the guarantee. For $n\ge4$, a disjoint
triangle and a tree on the other $n-3$ vertices satisfy that equation and
give the same gap. A tree on a single vertex has no edges.
\end{proof}

\subsection{The sharp asymptotic threshold for symmetric LPs}
\label{ssec:sharp-lp-threshold}

The preceding LP formulation attains factor $3/2$, whereas the original
signed certificate only excludes factors below $5/4$. To close this gap
asymptotically, we again use the triangle $S=\{1,2,3\}$ and
\cref{lem:outer-orbit}. We will construct signed tree weights of total one
with nonnegative sums on the required orbits, but with triangle value
$3(n-1)/n$. Since a spanning tree has at most two edges inside $S$, this
will force an approximation factor of at least $\frac32(1-1/n)$.
This alternative, somewhat more elaborate construction is needed only for the stronger inapproximability guarantee.

\paragraph{Constructing the signed weights.}
For this construction, let $n\ge6$. We use Hamilton paths, that is,
spanning trees that visit all $n$ vertices in a single path. Such a path
can be specified by an ordering $\sigma=(\sigma_1,\ldots,\sigma_n)$ of
the vertex labels. It is important to distinguish the \emph{positions}
$i$ in this ordering from the \emph{labels} $\sigma_i$ occupying them.
The three positions occupied by labels in $S$ form the set
$A_\sigma=\{i\in[n]:\sigma_i\in S\}$.

Let $C_n$ be the cycle on the positions $1,\ldots,n$, including the pair
$\{n,1\}$. This auxiliary cycle treats all positions alike; it is not the
spanning tree we will use. For each three-element set of positions $A$,
let $k(A)$ count its adjacent pairs on the cycle and set
\begin{equation}\label{eq:path-triple-weights}
 k(A)=|E(C_n)\cap\tbinom A2|,\qquad
 w_A=\frac{3}{n(n-4)}\left(k(A)-\frac4{n-2}\right).
\end{equation}
The subtraction is chosen so that the weights of triples containing any
fixed nonadjacent pair sum to zero. The prefactor makes the total weight
one. We verify both properties below. These are signed weights, not
probabilities: if $A$ has no adjacent pair, then $w_A<0$.

For a fixed $A$, there are $3!$ ways to place the labels in $S$ in its
three positions and $(n-3)!$ ways to place all other labels in the
remaining positions. Divide $w_A$ equally among these orderings. Thus
an ordering $\sigma$ receives weight
\[
 \nu_\sigma=\frac{w_{A_\sigma}}{3!(n-3)!}.
\]
Define $T(\sigma)$ to be the spanning path on $[n]$ with edges
$\{\sigma_i,\sigma_{i+1}\}$ for $i=1,\ldots,n-1$.
These edges connect the vertices occupying adjacent positions in the
ordering, not necessarily vertices with consecutive numerical labels.
The closing pair $\{\sigma_n,\sigma_1\}$ is not included. Finally, define
\begin{equation}\label{eq:path-tree-weights}
 \mu_T=\sum_{\sigma:\,T(\sigma)=T}\nu_\sigma.
\end{equation}
For a fixed labeled Hamilton path $T$, this sum has exactly two
terms: an ordering $\sigma$ and its reverse
$\sigma^{\mathrm{rev}}=(\sigma_n,\ldots,\sigma_1)$. Indeed, choosing either
endpoint as the first vertex forces the rest of the ordering; no other
ordering gives the same path. Reversal reflects the auxiliary cycle,
so it preserves $k(A)$ and $w_A$. Hence
$\mu_T=\nu_\sigma+\nu_{\sigma^{\mathrm{rev}}}=2\nu_\sigma$.
The $n!$ orderings therefore pair into $n!/2$ distinct labeled undirected
Hamilton paths. Trees that are not Hamilton paths receive weight zero. \Cref{fig:path-certificate} illustrates the
construction.

\begin{figure}[t]
 \centering
 \includegraphics[width=\linewidth]{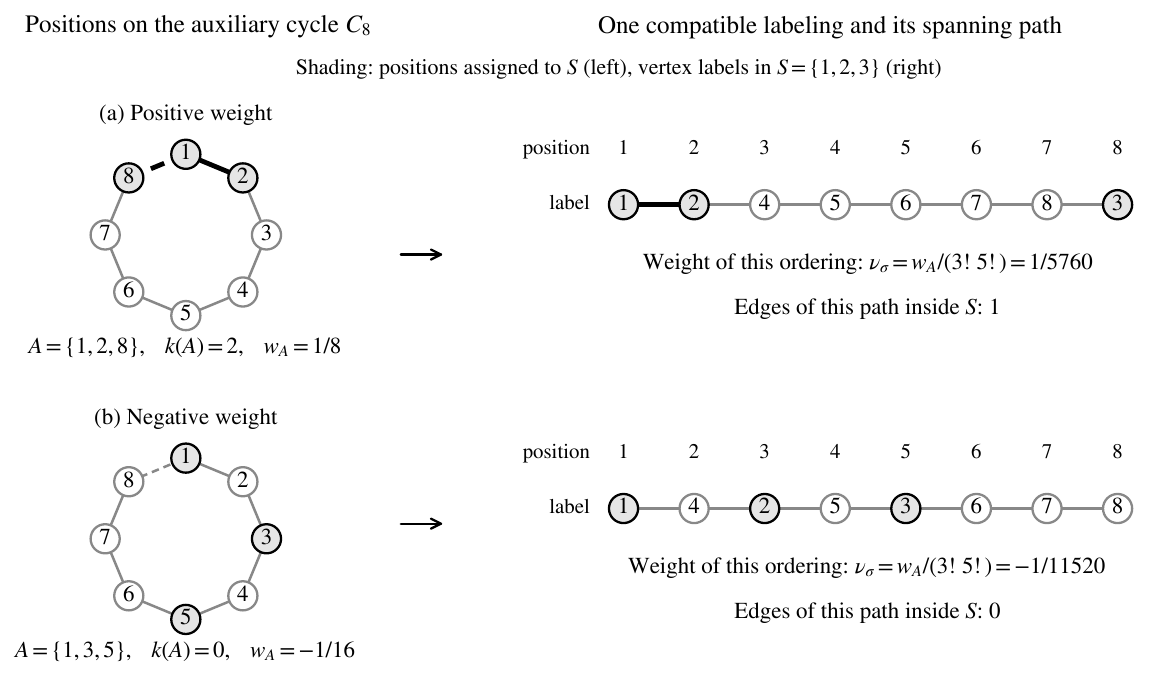}
 \caption{Two examples of the construction for $n=8$.
 The numbers on the left are positions; shaded positions are assigned
 the labels in $S=\{1,2,3\}$. On the right, vertex labels appear inside
 the circles and position numbers above them. The dashed closing pair
 $\{8,1\}$ counts in $k(A)$ but is not used in the path: in the top
 example $k(A)=2$, yet the path has only one edge inside $S$.
 Each displayed ordering receives weight $w_A/(3!5!)$; the tree weight
 $\mu_T$ also includes the reverse ordering.}
 \label{fig:path-certificate}
\end{figure}

\begin{lemma}[Position sums and fixed labels \leanverified]\label{lem:path-marginals}
For $n\ge6$, the triple weights $w_A$ satisfy $\sum_A w_A=1$. For any
position $i$, and any two distinct positions $i,j$, respectively,
\[
 t_i:=\sum_{A\ni i}w_A=\frac3n,\qquad
 p_{ij}:=\sum_{A\supseteq\{i,j\}}w_A=
 \begin{cases}
  3/n,&\{i,j\}\in E(C_n),\\
  0,&\text{otherwise.}
 \end{cases}
\]
If we prescribe the positions of at most two fixed vertex labels, the
sum of $\nu_\sigma$ over all orderings satisfying that prescription is
nonnegative.
\end{lemma}
\begin{proof}
\emph{Counting edge--triple pairs.}
We count pairs $(A,e)$ where $A$ is a three-element set of positions
and $e\in E(C_n)$ has both endpoints in $A$. Thus each triple $A$
contributes $k(A)$ pairs, not just one. Each of the $n$ cycle edges
can be extended to a triple in $n-2$ ways, by choosing a position
outside its two endpoints. Now fix $i$ and count only pairs with $i\in A$. If the cycle edge $e$
contains $i$, its endpoints already include $i$, and the third position
can be any of the other $n-2$ positions; there are two such cycle edges.
If $e$ does not contain $i$, the triple is forced to be $A=e\cup\{i\}$;
there are $n-2$ such cycle edges. Consequently,
\[
 \sum_A k(A)=n(n-2),\qquad
 \sum_{A\ni i}k(A)=2(n-2)+(n-2)=3(n-2).
\]
For a fixed pair $i\ne j$, if $i$ and $j$ are adjacent, the pair contributes its own edge
in all $n-2$ triples, together with the two other cycle edges incident
with its endpoints. If $i$ and $j$ are nonadjacent, the pair contributes the four cycle
edges incident with its endpoints. Thus
\[
 \sum_{A\supseteq\{i,j\}}k(A)=
 \begin{cases}(n-2)+2=n,&\{i,j\}\in E(C_n),\\
 4,&\text{otherwise.}\end{cases}
\]
There are $\binom n3$ triples in total, $\binom{n-1}2$ containing a
fixed position, and $n-2$ containing a fixed pair. Substituting these
counts into \eqref{eq:path-triple-weights} gives
\[
 \sum_A w_A
 =\frac{3}{n(n-4)}\left(n(n-2)-\frac4{n-2}\binom n3\right)=1,
\]
and the position sums follow by counting the constant subtraction
once for each admissible triple.

\emph{Triples containing one position.}
For a fixed position $j$, there are $\binom{n-1}{2}$ triples containing
$j$, since we choose the other two positions from the remaining $n-1$.
The edge--triple count above gives $\sum_{A\ni j}k(A)=3(n-2)$, so
\[
\begin{aligned}
 t_j
 &=\frac{3}{n(n-4)}
   \left(3(n-2)-\frac4{n-2}\binom{n-1}{2}\right)\\
 &=\frac{3}{n(n-4)}\bigl(3(n-2)-2(n-1)\bigr)
 =\frac3n.
\end{aligned}
\]

\emph{Triples containing two positions.}
For distinct positions $i,j$, there are $n-2$ admissible triples,
one for each choice of the third position. Thus the constant subtraction
totals $\frac4{n-2}(n-2)=4$. Using the preceding edge--triple counts,
\[
\begin{aligned}
 p_{ij}
 &=\frac3{n(n-4)}
   \left(\sum_{A\supseteq\{i,j\}}k(A)-4\right)\\
 &=\begin{cases}
  3/n,&\{i,j\}\in E(C_n),\\
  0,&\text{otherwise.}
 \end{cases}
\end{aligned}
\]
In the adjacent case, the expression in parentheses is $n-4$;
in the nonadjacent case, it is $4-4=0$. This explains the choice of
$4/(n-2)$ in \eqref{eq:path-triple-weights}: it makes the signed sum
for each nonadjacent pair exactly zero.

\emph{Prescribing vertex labels.}
Now fix distinct labels $u,v$ and distinct positions $i,j$. The total
signed weight of orderings with $\sigma_i=u$ and $\sigma_j=v$ is
\[
 \sum_{\sigma:\,\sigma_i=u,\ \sigma_j=v}\nu_\sigma
 =\begin{cases}
  \frac{p_{ij}}6,&u,v\in S,\\[4pt]
  \frac{t_i-p_{ij}}{3(n-3)},&u\in S,\ v\notin S,\\[4pt]
  \frac{t_j-p_{ij}}{3(n-3)},&u\notin S,\ v\in S,\\[4pt]
  \frac{1-t_i-t_j+p_{ij}}{(n-3)(n-4)},&u,v\notin S.
 \end{cases}
\]
To obtain these formulas, first sum $w_A$ over the compatible
position triples $A$, then multiply by the fraction of their label
assignments satisfying $\sigma_i=u$ and $\sigma_j=v$. Within each case,
this fraction is the same for every compatible triple.

Case: Both labels lie in $S$. We need $i,j\in A$, so the compatible triples have total weight $p_{ij}$.
For each such triple there are $3\cdot2=6$ ordered choices of distinct
$S$-labels for these two positions, all occurring equally often among
its assignments. Exactly one is the prescribed pair $(u,v)$, giving
the fraction $1/6$ and hence the weight $p_{ij}/6$.

Case: Exactly one label lies in $S$.
If $u\in S$ and $v\notin S$, the triple must contain $i$ but not $j$,
so the compatible triples have total weight $t_i-p_{ij}$.
A fraction $1/3$ of assignments
puts $u$ at $i$, and independently a fraction $1/(n-3)$ puts $v$ at
$j$, since the labels are assigned from disjoint sets. Multiplying
these fractions gives the denominator $3(n-3)$. The other mixed case
interchanges $i$ and $j$, giving the numerator $t_j-p_{ij}$; the two
values agree because $t_i=t_j$.

Case: Both labels lie outside $S$.
The triple must avoid both $i$ and $j$. By inclusion--exclusion, the
compatible triples have total weight $1-t_i-t_j+p_{ij}$. For each
such triple, there are $n-3$ choices of outside label for position $i$
and then $n-4$ for position $j$. These ordered choices occur equally
often, so the prescribed pair accounts for a fraction
$1/[(n-3)(n-4)]$ of the assignments, giving the last formula.

These fractions count equally weighted assignments for a fixed $A$;
they do not assume that the signed triple weights $w_A$ are probabilities. All these quantities are nonnegative: $p_{ij}$ is either $0$ or $3/n$,
and $1-t_i-t_j=1-6/n\ge0$. Fixing just one label at a given position
has weight $t_i/3=1/n$ if it lies in $S$, and
$(1-t_i)/(n-3)=1/n$ otherwise. With no prescribed labels, the total
weight is one. An impossible prescription, such as two distinct labels
in the same position, has weight zero.
\end{proof}

\paragraph{From fixed labels to orbit sums.}
The preceding lemma concerns orderings with specified labels in specified
positions. This is unfortunately only half of what has to be done. In order to
apply \cref{lem:outer-orbit}, we must instead show
nonnegativity on orbits of trees under even permutations fixing those
labels. The next lemma makes both steps explicit, providing the necessary link to apply our criterion.

\begin{lemma}[Orbit sums and triangle value \leanverified]\label{lem:path-orbits}
Let $n\ge7$, and let $(\mu_T)_T$ be the signed tree weights defined in
\eqref{eq:path-tree-weights}. For every $W\subseteq[n]$ with $|W|\le2$,
each $A([n]\setminus W)$-orbit $\mathcal O$ of spanning trees satisfies
$\sum_{T\in\mathcal O}\mu_T\ge0$. Moreover,
\begin{equation}\label{eq:path-triangle-value}
 \sum_T\mu_T=1 \qquad \text{and} \qquad
 \sum_T\mu_T|T\cap E(S)|=\frac{3(n-1)}n.
\end{equation}
\end{lemma}
\begin{proof}
\emph{Even permutations of labels.}
Fix $W\subseteq[n]$ with $|W|\le2$. Let $G_W$ be the group of all
permutations fixing every label in $W$, and let
$H_W=A([n]\setminus W)$ be its subgroup of even permutations.
Relabeling acts on an ordering by
$g\sigma=(g(\sigma_1),\ldots,g(\sigma_n))$.
A $G_W$-orbit $\mathcal C$ consists exactly of the orderings with
prescribed positions for the labels in $W$: the permutation taking the
label in each position of one such ordering to the label in that
position of another fixes $W$. Hence $\mathcal C$ has nonnegative total
weight by \cref{lem:path-marginals}.

The action on orderings is free: if $g\sigma=\sigma$, then $g$ fixes
every label and is the identity. Since $H_W$ has index two in $G_W$,
the orbit $\mathcal C$ splits into two $H_W$-orbits $\mathcal C_+$ and
$\mathcal C_-$, obtained by even and odd relabelings of a chosen
representative. To show that these two orbits have equal total weight,
choose two distinct labels outside $S\cup W$; there are at least
$n-5\ge2$ such labels. Their transposition $\tau$ fixes $W$ and leaves
$A_\sigma$ unchanged, so $\nu_{\tau\sigma}=\nu_\sigma$.
Composition with this odd transposition changes the parity of the
relabeling. Thus $\sigma\mapsto\tau\sigma$ is a weight-preserving
bijection between $\mathcal C_+$ and $\mathcal C_-$, giving
\[
 \sum_{\sigma\in\mathcal C_+}\nu_\sigma
 =\sum_{\sigma\in\mathcal C_-}\nu_\sigma
 =\frac12\sum_{\sigma\in\mathcal C}\nu_\sigma\ge0.
\]

\emph{Passing from orderings to trees.}
The map $\sigma\mapsto T(\sigma)$ commutes with vertex relabeling:
$T(g\sigma)=gT(\sigma)$. Consequently, for any $H_W$-orbit
$\mathcal O$ of trees, its preimage
$\{\sigma:T(\sigma)\in\mathcal O\}$ is a disjoint union of
$H_W$-orbits of orderings. Summing their nonnegative total weights and
using \eqref{eq:path-tree-weights} gives
\[
 \sum_{T\in\mathcal O}\mu_T
 =\sum_{\sigma:\,T(\sigma)\in\mathcal O}\nu_\sigma\ge0.
\]
The two orderings of each undirected path are already included in this
sum, so no extra factor of two is needed. If the tree orbit contains
no Hamilton path, its preimage is empty and its weight is zero.
Also, distributing each $w_A$ among its orderings and then adding their
contributions to trees preserves the total weight:
\[
 \sum_T\mu_T=\sum_\sigma\nu_\sigma=\sum_A w_A=1.
\]

\emph{Evaluating the triangle objective.}
For an ordering $\sigma$, the path edge
$\{\sigma_i,\sigma_{i+1}\}$ lies in $E(S)$ exactly when
$\{i,i+1\}\subseteq A_\sigma$. There are $n-1$ consecutive pairs of
positions; the closing pair $\{n,1\}$ is not a path edge. Grouping
orderings with the same position set $A_\sigma=A$ recovers total
weight $w_A$. Therefore
\[
\begin{aligned}
 \sum_T\mu_T|T\cap E(S)|
 &=\sum_\sigma\nu_\sigma
   \sum_{i=1}^{n-1}\mathbf1_{\{i,i+1\}\subseteq A_\sigma}\\
 &=\sum_{i=1}^{n-1}\sum_{A\supseteq\{i,i+1\}}w_A\\
 &=\sum_{i=1}^{n-1}p_{i,i+1}=\frac{3(n-1)}n,
\end{aligned}
\]
where the last equality uses \cref{lem:path-marginals} for each cycle
edge $\{i,i+1\}$. This is a signed sum, not an expectation under a
probability distribution; negative weights allow it to exceed the tree
optimum two.
\end{proof}

\paragraph{The approximation lower bound.}
For $n\ge13$, these weights satisfy the normalization and orbit-sum
conditions in \cref{lem:outer-orbit}. Applying that criterion and then
testing the triangle objective gives the promised inapproximability factor.

\begin{theorem}[Symmetric LP inapproximability {\leanassuming[; lower bounds]}]\label{thm:approx-lp-threshold}
Let $n\ge13$. If an $S_n$-symmetric LP outer relaxation of
$P_{\ST}(K_n)$ has fewer than $\binom n3$ facets and has factor $\beta$,
then
\[
 \beta\ge\frac32\left(1-\frac1n\right).
\]
Consequently, for every fixed $1\le\beta<3/2$ and all sufficiently large
$n$, the minimum size of a symmetric LP formulation with factor $\beta$
is $\Theta(n^3)$. At factor $3/2$, size $3m+n$ suffices.
\end{theorem}
\begin{proof}
Let $p:Q\to R$ be the formulation, with fewer than $\binom n3$ facets
in $Q$. By \cref{lem:outer-orbit,lem:path-orbits}, the signed affine
combination
\[
 y=\sum_T\mu_T\chi^T
\]
belongs to $R$ and satisfies $y(E(S))=3(n-1)/n$. Take the nonnegative
objective $c=\mathbf1_{E(S)}$, which gives weight one to triangle edges
and zero to all other edges. Every spanning tree uses at most two
triangle edges, since all three would form a cycle. The bound is attained
by the path $1$--$2$--$3$ with every other vertex attached as a leaf at
vertex $1$. Thus \eqref{eq:approx-factor} gives
\[
 \frac{3(n-1)}n=c^\top y
 \le\beta\max_{T\text{ spanning tree}}c^\top\chi^T
 =2\beta,
\]
proving the bound for every finite $n$.

For any fixed $\beta<3/2$, this rules out fewer than $\binom n3$ facets
whenever $n\ge13$ and $n>3/(3-2\beta)$. This is the cubic lower bound.
The exact symmetric formulation in \cref{prop:symmetric-upper} complements this lower bound with a cubic upper bound, since an exact formulation obviously also satisfies every approximation guarantee with $\beta\ge1$. At factor $3/2$,
\cref{prop:approx-lp-upper} instead gives size $3m+n=O(n^2)$.
\end{proof}

The inapproximability threshold is sharp in this asymptotic sense: every \emph{fixed}
factor below $3/2$ requires cubic size, while factor $3/2$ admits
quadratic size. Note that the bound for finite $n$ does not need to be tight and the best approximation factor achievable under the stated symmetry and facet constraints could be strictly larger.

\section{\texorpdfstring{Exact semidefinite formulations of size $O(n^2)$}{Exact semidefinite formulations of quadratic size}}
\label{ssec:exact-sdp}

To complement the lower bounds for (approximate) symmetric LPs we will now derive an exact SDP lift of the spanning tree polytope of quadratic size. To this end, we will encode all subtour slacks as quadratic forms in one PSD matrix. We first give an exact symmetric formulation of size $O(n^2)$ and then further optimize its size, leading to two smaller formulations (one symmetric the other not) by reusing the same quadratic form but exploiting redundancy and breaking symmetry.

Our notions of (symmetric) semidefinite extended formulations are standard as to be found, e.g., in \citep{GouveiaParriloThomas2013,BPR2015,BPR2015jour,BDP2013,BDP2013jour}. Let $\mathbb S^k$ be the set of real symmetric $k\times k$ matrices and
$\mathbb S_+^k$ for their positive semidefinite cone. An \emph{SDP lift}
of a convex set $R$ is a representation
\[
 R=p(L\cap\mathbb S_+^k),
\]
where $L\subseteq\mathbb S^k$ is an affine subspace and $p$ is an affine
map. As customary, the size of an SDP lift is its PSD matrix order,
not its scalar-variable count. A symmetric $k\times k$ matrix has
$k(k+1)/2$ scalar coordinates before the affine equations are imposed.
Equivalently, one may use a product of PSD cones and count the
sum of the block orders: block-diagonal embedding gives a single cone
of that order. Every scalar inequality counts as a block of order one;
affine equations do not count.

\begin{definition}[Semidefinite extension complexity]\label{def:sdp-xc}
The \emph{semidefinite extension complexity} of $R$ is
\[
 \xc_{\mathrm{psd}}(R)
 :=\min\{k:R=p(L\cap\mathbb S_+^k)
       \text{ for an affine subspace }L\text{ and an affine map }p\}.
\]
We set the value to $+\infty$ if no such lift exists.
\end{definition}

For a symmetric SDP lift, the group action must preserve the feasible
set and commute with the projection. Our first construction uses
simultaneous permutations of matrix rows and columns. Its compression
uses orthogonal congruences $Y\mapsto Q_gYQ_g^{\mathsf T}$, where
$Q_g^{\mathsf T}Q_g=I$ and $Q_{gh}=Q_gQ_h$. Both are equivariant cone
lifts in the usual sense \citep{GouveiaParriloThomas2013}.

\paragraph{Parent edges and the subtour slack.} The following simple observation is the key ingredient for the SDP formulations.
For a spanning tree $T$ rooted at $r$ let $p_{T,r}(v)$ be the next vertex
on the unique path from $v \neq r$ to $r$. Further, let $(f_e)_{e\in E}$ be the
standard orthonormal basis of $\R^E$, i.e., incidence vectors. We define two key quantities
\begin{equation}\label{eq:tree-gram-vectors}
 a_e(T)=\mathbf 1_{\{e\in T\}}f_e \qquad \text{and} \qquad
 \qquad b_{rv}(T)=f_{\{v,p_{T,r}(v)\}} \quad \text{for } v \neq r.
\end{equation}
Note that $r$ here is the root and not the parent of $v$. In fact, the parent edge is specified by the value of $b_{rv}(T)$. For a fixed root $r$, the $n-1$ parent vectors $b_{rv}(T)$ with $v \neq r$ are distinct edge basis vectors, since every tree edge has exactly one child endpoint.

The key is now the following argument: If $r\in S$, each tree edge internal to $S$ occurs once in each of the
following two sums and cancels. The remaining vectors correspond to
parent edges leaving $S$, one for each component of $T[S]$ other than
the component containing $r$. Consequently,
\begin{equation}\label{eq:tree-gram-slack}
 |S|-1-|T\cap E(S)|
 =\left\|\sum_{v\in S\setminus\{r\}}b_{rv}(T)
          -\sum_{e\in E(S)}a_e(T)\right\|^2.
\end{equation}

This rooted-tree counting interpretation is classical: it underlies the
$O(nm)$ LP formulation described by
\citet[Section~5.2]{FaenzaFioriniGrappeTiwary2015}, which recovers
Martin's bound.
Here we use the vectors in \eqref{eq:tree-gram-vectors} as the rows of a
single Gram matrix, shared by all roots.
For example, on the path $1$--$2$--$3$--$4$, with $r=1$ and
$S=\{1,3,4\}$, the two sums are $f_{23}+f_{34}$ and $f_{34}$.
Their difference is $f_{23}$, with squared norm one, exactly the subtour slack.

\paragraph{The symmetric formulation.}
We use the following set of matrix indices, which we formally make a disjoint union of edges labels of the form $\{u,v\} \in E$ and root-child labels of the form $(r,v)$ with $r \neq v$:
\[
 \mathcal L=E\mathbin{\dot\cup}\{(r,v):r,v\in[n],\ r\ne v\}.
\]
There are $m+n(n-1)=3m$ indices. For $Y\in\mathbb S_+^{\mathcal L}$,
we project onto the edge coordinates via $x_e=Y_{e,e}$. The SDP formulation is given by the following affine equations together with $Y \succeq 0$:
\begin{align}
 \sum_{e\in E}x_e&=n-1,\label{eq:exact-sdp-cardinality}\\
 Y_{e,f}&=0
   && e,f\in E,\ e\ne f,\label{eq:exact-sdp-edges}\\
 Y_{(r,v),(r,w)}&=\delta_{vw}
   && r\in[n],\ v,w\ne r,\label{eq:exact-sdp-parents}\\
 Y_{e,(r,v)}&=0
   && r\ne v,\ e\in E,\ v\notin e,\label{eq:exact-sdp-incidence}\\
 \sum_{v\ne r}Y_{e,(r,v)}&=x_e
   && r\in[n],\ e\in E.\label{eq:exact-sdp-conservation}
\end{align}
The symbol $\delta_{vw}$ is one if $v=w$ and zero otherwise.
The edge block is diagonal, and the parent block for each root $r \in [n]$ is an
identity matrix.
The incidence equations \eqref{eq:exact-sdp-incidence} force the edge-parent Gram entries to vanish whenever the edge $e$ is not incident to the child vertex $v$.
For each fixed root $r \in [n]$, the conservation equations \eqref{eq:exact-sdp-conservation} require the edge-parent Gram entries associated with each edge $e$ to sum to $x_e$. Together with \eqref{eq:exact-sdp-incidence} these become $Y_{e,(r,u)} + Y_{e,(r,v)} = x_e$ for $e=\{u,v\}$ and $r \not\in e$ and for $e=\{r,v\}$
the conservation equation reduces to $Y_{e,(r,v)} = x_e$. All equations are
linear in $Y$ after substituting $x_e=Y_{e,e}$. Note that the vectors $a_e(T)$ and $b_{rv}(T)$ from above associated with a given tree $T$ are only used in the proof and are not part of the SDP formulation.

\begin{theorem}[An exact symmetric SDP of quadratic size \leanverified]\label{thm:exact-sdp-symmetric}
For $n\ge2$, the projection $x_e=Y_{e,e}$ of
$Y\succeq0$ satisfying
\eqref{eq:exact-sdp-cardinality}--\eqref{eq:exact-sdp-conservation}
is exactly $P_{\ST}(K_n)$. This is an $S_n$-symmetric SDP formulation
of size $3m$, with vertex relabeling acting by simultaneous permutations
of matrix rows and columns.
\end{theorem}

\begin{proof}
\emph{Every tree lifts.} First we show that every spanning tree $T$ has a lift $Y_T \succeq 0$ satisfying \eqref{eq:exact-sdp-cardinality}--\eqref{eq:exact-sdp-conservation}. Let $T$ be any spanning tree. We form a matrix $H_T$ whose rows are $a_e(T)$ and $b_{rv}(T)$, and set $Y_T=H_TH_T^{\mathsf T}$. It is PSD and projects to $\chi^T$ as $x_e = (Y_T)_{e,e} = \norm{a_e(T)}^2 = \mathbf 1_{\{e\in T\}}$.
Distinct edge vectors are orthogonal, and the parent vectors for each
root are orthonormal. Incidence zeros in \eqref{eq:exact-sdp-incidence} follow from the definition of a parent edge. Finally, every tree edge occurs once among the parent
vectors for each root, giving \eqref{eq:exact-sdp-conservation}.
Since $T$ has $n-1$ edges, \eqref{eq:exact-sdp-cardinality} also holds.
Thus $Y_T$ satisfies all the equations. Convex combinations of these
matrices lift every point of $P_{\ST}(K_n)$.

\emph{Every feasible matrix projects into the polytope.} Next we show that an arbitrary $Y \succeq 0$ satisfying \eqref{eq:exact-sdp-cardinality}--\eqref{eq:exact-sdp-conservation} projects into the spanning tree polytope.
Let $\alpha_e$ and $\beta_{rv}$ denote the coordinate unit vectors in $\mathbb R^{\mathcal L}$ at indices $e$ and $(r,v)$, respectively. For a nonempty
$S\subseteq[n]$ and any $r\in S$, put
\[
 A_S=\sum_{e\in E(S)}\alpha_e
  \qquad \text{and}
 \qquad B_{S,r}=\sum_{v\in S\setminus\{r\}}\beta_{rv}.
\]

By edge and parent orthogonality, i.e., \eqref{eq:exact-sdp-edges}, \eqref{eq:exact-sdp-parents},
and the projection $x_e=Y_{e,e}$, we have
\begin{align*}
A_S^{\mathsf T}YA_S
&=\sum_{e,f\in E(S)}Y_{e,f}
=\sum_{e\in E(S)}Y_{e,e}
=x(E(S)),\\
B_{S,r}^{\mathsf T}YB_{S,r}
&=\sum_{v,w\in S\setminus\{r\}}Y_{(r,v),(r,w)}
=\sum_{v,w\in S\setminus\{r\}}\delta_{vw}
=|S|-1,
\end{align*}
where the last equality uses $r\in S$. For an internal edge $e\in E(S)$, every term with $v\notin S$
in \eqref{eq:exact-sdp-conservation} vanishes by
\eqref{eq:exact-sdp-incidence}. Hence
\[
 A_S^{\mathsf T}YB_{S,r}
 =\sum_{e\in E(S)}\sum_{v\in S\setminus\{r\}}Y_{e,(r,v)}
 =x(E(S)).
\]
Expanding the difference proves the identity
\begin{equation}\label{eq:exact-sdp-slack}
 0\le (B_{S,r}-A_S)^{\mathsf T}Y(B_{S,r}-A_S)
       =|S|-1-x(E(S)).
\end{equation}
Thus every subtour inequality holds. Also $x_e = Y_{e,e} = \alpha_e^{\mathsf T}Y \alpha_e \geq 0$ by
positive semidefiniteness, and $x(E)=n-1$ is imposed explicitly.
Edmonds' description now gives $x\in P_{\ST}(K_n)$.

\emph{Symmetry and size.}
For $\sigma\in S_n$, permute the indices by
$e\mapsto\sigma(e)$ and $(r,v)\mapsto(\sigma(r),\sigma(v))$.
If $P_\sigma$ is the corresponding permutation matrix, the action is
$Y\mapsto P_\sigma YP_\sigma^{\mathsf T}$.
It preserves positive semidefiniteness, permutes each displayed family
of equations, and induces $x\mapsto\sigma x$ under projection.
The matrices $P_\sigma$ form a group representation. There is one PSD
block of order $3m$ and no other cone constraints.
\end{proof}

\paragraph{Compression while preserving symmetry.} We will now further reduce the size by exploiting the following identity: For every root $r$, the tree vectors satisfy
\begin{equation}\label{eq:exact-sdp-root-relation}
 \sum_{v\ne r}b_{rv}(T)=\sum_{e\in E}a_e(T).
\end{equation}

Using these relations, we can reduce the PSD matrix order by $n$
while preserving $S_n$-symmetry. In the compressed coordinates, vertex
relabeling acts by orthogonal congruences, which need not be coordinate
permutations.

\begin{corollary}[A compressed symmetric SDP \leanverified]\label{cor:exact-sdp-symmetric-compressed}
For $n\ge2$, there is an exact $S_n$-symmetric SDP formulation of
$P_{\ST}(K_n)$ with one PSD block of order $3m-n$.
Vertex relabeling acts by orthogonal congruences.
\end{corollary}
\begin{proof}
In $\R^{\mathcal L}$ define
\[
 k_r=\sum_{v\ne r}\beta_{rv}-\sum_{e\in E}\alpha_e,
 \qquad K=\operatorname{span}\{k_r:r\in[n]\}.
\]
The $k_r$ are linearly independent: the coordinate $(r,v)$, for any
$v\ne r$, occurs in $k_r$ and in no other $k_s$. Thus $\dim K=n$.
Choose a fixed matrix $U$ whose columns are an orthonormal basis of
$K^\perp$. Substitute
\begin{equation}\label{eq:exact-sdp-symmetric-compression}
 Y=UZU^{\mathsf T},\qquad Z\in\mathbb S_+^{3m-n}
\end{equation}
in the preceding equations and projection. This gives affine equations
in $Z$ and projects into $P_{\ST}(K_n)$ by \cref{thm:exact-sdp-symmetric}.
For each tree, \eqref{eq:exact-sdp-root-relation} gives
$H_T^{\mathsf T}k_r=0$, so the range of $Y_T$ lies in $K^\perp$.
Therefore $Z_T=U^{\mathsf T}Y_TU\succeq0$ satisfies
$UZ_TU^{\mathsf T}=Y_T$. Every tree still lifts, proving exactness.

Since $P_\sigma k_r=k_{\sigma(r)}$, both $K$ and $K^\perp$ are
invariant. The matrices
\[
 Q_\sigma=U^{\mathsf T}P_\sigma U
\]
are orthogonal and satisfy $Q_{\sigma\tau}=Q_\sigma Q_\tau$.
The action $Z\mapsto Q_\sigma ZQ_\sigma^{\mathsf T}$ corresponds
under \eqref{eq:exact-sdp-symmetric-compression} to the previous
permutation action on $Y$. It therefore preserves the affine slice and
commutes with the projection. The $Q_\sigma$ need not be permutation
matrices; the claimed symmetry is by orthogonal congruences.
\end{proof}

\paragraph{A smaller formulation without symmetry.}
The proof of \eqref{eq:exact-sdp-slack} only needs one root in each
nonempty set $S$. Fix the order $1<\cdots<n$ and always choose $r=\min S$.
Then only parent indices $(r,v)$ with $r<v$ are needed. There are $m$
such indices, reducing the matrix order from $3m$ to $2m$ before compression.

\begin{theorem}[An exact asymmetric SDP \leanverified]\label{thm:exact-sdp-asymmetric}
For $n\ge2$, the spanning tree polytope has an exact SDP formulation
with one PSD block of order $2m$, reducible to order $2m-1$. In particular,
\[
 \xc_{\mathrm{psd}}(P_{\ST}(K_n))\le 2m-1=n(n-1)-1.
\]
\end{theorem}
\begin{proof}
\emph{The order-$2m$ formulation.}
Keep the $m$ edge indices and the $m$ parent indices $(r,v)$ with $v>r$.
For a tree witness, this condition restricts the label of the child $v$,
not the label of its parent $p_{T,r}(v)$.
Keep \eqref{eq:exact-sdp-cardinality}, \eqref{eq:exact-sdp-edges},
and the projection $x_e=Y_{e,e}$. Impose
\eqref{eq:exact-sdp-parents} and \eqref{eq:exact-sdp-incidence} only
for retained parent indices. Replace the conservation equations by
\begin{equation}\label{eq:exact-sdp-ordered-conservation}
 \sum_{v>r}Y_{e,(r,v)}=x_e
 \qquad \text{ with } r\in[n],\ e\in E(\{r,r+1,\ldots,n\}).
\end{equation}
Thus this root's conservation equations are imposed only on edges with
endpoints both at least $r$.

Let $\widehat H_T$ be the restriction of the original tree factor $H_T$
to these $2m$ rows. Its Gram matrix inherits the edge, parent and incidence
equations. For conservation, root the whole tree $T$ at $r$. Each tree
edge with both endpoints at least $r$ has exactly one child endpoint;
that child is not $r$, so its label is greater than $r$. The corresponding
parent row is therefore retained and contributes once to the sum in
\eqref{eq:exact-sdp-ordered-conservation}. Nontree edges have zero edge
rows. Hence $\widehat H_T\widehat H_T^{\mathsf T}$ is feasible and
projects to $\chi^T$; convex combinations lift the entire polytope.

Conversely, let $Y$ be feasible for this reduced system. For nonempty
$S$, choose $r=\min S$ and define $A_S$ and $B_{S,r}$ using the coordinate
unit vectors of the reduced index set. Every $v\in S\setminus\{r\}$
satisfies $v>r$, so all parent indices in $B_{S,r}$ are retained.
Every $e\in E(S)$ also has both endpoints at least $r$, so its conservation
equation is present. Incidence zeros again remove the terms outside $S$.
The same three pairings as in the proof of \cref{thm:exact-sdp-symmetric}
therefore give \eqref{eq:exact-sdp-slack}. Together with nonnegativity
and cardinality, this proves that the projection is exactly $P_{\ST}(K_n)$.

\emph{Compression to order $2m-1$.}
Root $1$ retains all of its parent indices. Its full-root relation
\eqref{eq:exact-sdp-root-relation} therefore allows us to recover row
$(1,n)$ from the other rows:
\[
 b_{1n}(T)=\sum_{e\in E}a_e(T)-\sum_{v=2}^{n-1}b_{1v}(T).
\]
Now choose coordinate unit vectors $\alpha_e$ and $\beta_{rv}$ in
$\R^{2m-1}$ for the remaining indices, excluding $(1,n)$, and define
\begin{equation}\label{eq:exact-sdp-deleted-coordinate}
 \beta_{1n}=\sum_{e\in E}\alpha_e-\sum_{v=2}^{n-1}\beta_{1v}.
\end{equation}
This last vector is a fixed linear combination, not an additional
coordinate or optimization variable.
Let $C\in\R^{2m\times(2m-1)}$ have rows $c_i^{\mathsf T}$ indexed by
the old $2m$ labels, where $c_e=\alpha_e$ and $c_{(r,v)}=\beta_{rv}$.
For $Z\in\mathbb S_+^{2m-1}$ substitute
\[
 Y=CZC^{\mathsf T},\qquad Y_{i,j}=c_i^{\mathsf T}Zc_j
\]
in every equation of the order-$2m$ system. The projection becomes
$x_e=\alpha_e^{\mathsf T}Z\alpha_e=Z_{e,e}$.
Since $C$ is fixed, all these equations are affine in $Z$. Every feasible
$Z$ reconstructs a PSD matrix $Y$ satisfying the order-$2m$ system,
so its projection lies in $P_{\ST}(K_n)$.

Conversely, delete row $(1,n)$ from the already restricted factor
$\widehat H_T$ to obtain $\widetilde H_T$. The root-$1$ relation gives
$C\widetilde H_T=\widehat H_T$. Consequently,
\[
 Z_T=\widetilde H_T\widetilde H_T^{\mathsf T}\succeq0,
 \qquad CZ_TC^{\mathsf T}=\widehat H_T\widehat H_T^{\mathsf T}.
\]
Thus $Z_T$ is feasible and projects to $\chi^T$. Convex combinations
again lift the entire polytope, proving exactness after compression.
\end{proof}

The ordered choice of parent indices is not preserved by arbitrary
vertex relabelings; no symmetry is claimed for this smaller formulation.
In all versions, the coefficient vectors are fixed and edge
nonnegativity follows from the single PSD constraint. Although the
matrix orders are $O(n^2)$, the numbers of scalar entries and displayed
affine equations are $O(n^4)$. This distinction matters when comparing LP and SDP formulation
sizes.

\begin{corollary}[Separation between symmetric LPs and SDPs {\leanassuming[; LP lower bound]}]\label{cor:approx-separation}
For every fixed $1\le\beta<3/2$ and all sufficiently large $n$,
symmetric LP formulations with factor $\beta$ require $\Theta(n^3)$
inequalities, whereas an exact symmetric SDP formulation has size
$O(n^2)$. In particular, the separation holds for exact formulations.
\end{corollary}
\begin{proof}
Combine \cref{thm:approx-lp-threshold,thm:exact-sdp-symmetric}.
An exact formulation also satisfies every approximation guarantee
with $\beta\ge1$.
\end{proof}

It is tempting to seek a comparably small LP formulation by
preserving the feasible Gram witnesses of these SDPs through an affine map.
This however is not possible as the following obstruction shows.

\begin{remark}[Limits of LP conversion]\label{rem:lp-refinement}
The quadratic SDP size does not imply a comparably small LP formulation.
More precisely it can be shown that, if an LP polyhedron admits an affine map into the feasible
set of the compressed asymmetric Gram lift in
\cref{thm:exact-sdp-asymmetric}, such that the composite projection has
image $P_{\ST}(K_n)$, then the LP requires $\Omega(n^4)$ inequalities.
This obstruction concerns preserving the particular SDP lift, not
unrestricted LP extension complexity; it therefore does not contradict
the classical cubic-size LP formulations.
\end{remark}

For exact lifts, \citet[Theorems~1.2 and~4.9]{GrandeSanyal2017} imply that the semidefinite extension complexity of $P_{\ST}(K_n)$ is at least $\binom n2+1$ for $n\ge4$. Together with \cref{thm:exact-sdp-symmetric,thm:exact-sdp-asymmetric},
this shows that the exact SDP extension complexity is $\Theta(n^2)$, both
without symmetry and under either of the symmetry requirements used
above, i.e.,
\[
 m+1\le\xc_{\mathrm{psd}}(P_{\ST}(K_n))\le2m-1
 \qquad(n\ge4).
\]
We will next derive a lower bound also for approximate SDP formulations under the additional assumption that $R \subseteq \R_+^E$. This is really an additional assumption and restricts the potential approximations considered; at the same time we do not require symmetry for the lower bound. More specifically, under this additional assumption, for any fixed $\beta \ge 1$, the minimum SDP size
among nonnegative outer relaxations with factor at most $\beta$ is
$\Theta(n^2)$. In particular, basically the approximation can be arbitrarily bad and the structure of $P_{\ST}(K_n)$ together with the nonnegativity inequalities forces the lower bound; we include the bound here for the sake of completeness.

\begin{proposition}[Lower bound for approximate SDP lifts \leanverified]\label{prop:nonnegative-psd-order}
If $P_{\ST}(K_n)\subseteq R\subseteq\R_+^E$, every real PSD lift of $R$
has size at least
\[
 m-n+1=\binom{n-1}{2}.
\]
\end{proposition}
\begin{proof}
Fix a spanning tree $T_0$ and list its $q=m-n+1$ nonedges as
$e_1,\ldots,e_q$. For each $j$, add $e_j$ to $T_0$ and remove one original
edge from the resulting cycle to obtain a tree $T_j$. The valid coordinate
inequalities $x_{e_i}\ge0$ have the identity slack submatrix
\[
 x_{e_i}(\chi^{T_j})=\delta_{ij}.
\]
A PSD lift of size $r$ gives a PSD factorization of this matrix by
matrices of order $r$ \citep{GouveiaParriloThomas2013}. More explicitly,
restrict the lift to its minimal PSD face to obtain a strictly feasible
conic system on that face and apply conic duality to each coordinate
inequality. Each is tight at $T_0$, so the constant dual offset is zero.
A product of cones is embedded in one cone by block-diagonal matrices.

For completeness, suppose $\operatorname{Tr}(A_i B_j)=\delta_{ij}$ with
positive semidefinite matrices of order $r$. Choose
$u_i\in\operatorname{range}(A_i)$ with $B_i u_i\ne0$. For $i\ne j$,
zero trace implies $B_j u_i=0$. Applying $B_j$ to any linear dependence
among the $u_i$ shows that its $j$th coefficient is zero. The vectors
are therefore linearly independent, and $r\ge q$.
\end{proof}

\section*{AI use statement}
Generative AI tools assisted throughout the research process with formulation and revision of mathematical arguments, implementation of verification and plotting scripts, formalization, and manuscript editing. The associated checks include symbolic calculations, numerical and adversarial tests, and machine-checked proofs for the results covered by the formalization. The author has reviewed the mathematical arguments and verification evidence and takes responsibility for the final content, including material developed with the aid of generative AI.

\section*{Acknowledgments}
The author would like to thank Kanstantsin Pashkovich and Christoph Hertrich for the recent discussions on extended formulations that prompted the author to revisit the question as well as Gábor Braun for comments on an earlier draft. Research reported in this paper was partially supported by the Berlin Mathematics Research Center MATH$^+$ (EXC-2046/2, project ID 390685689), funded by the Deutsche Forschungsgemeinschaft (DFG, German Research Foundation) under Germany's Excellence Strategy.  This work has been partially supported by OpenAI's Academic Researchers program.

\bibliographystyle{plainnat}
\bibliography{refs}

\end{document}